\documentclass[a4paper,11pt,oneside]{amsart}
\usepackage[top=2.2cm,left=2.8cm,right=2.8cm,bottom=2.8cm]{geometry}

\usepackage[foot]{amsaddr}

\usepackage{hyperref}
\hypersetup{
    colorlinks,
    citecolor=teal,
    filecolor=black,
    linkcolor=blue,
    urlcolor=blue
}
\hypersetup{linktocpage}

\usepackage{cite}
\usepackage{graphicx}
\usepackage{qtree}

\usepackage{amsbsy}
\usepackage{latexsym}
\usepackage{amsfonts}
\usepackage{amssymb}
\usepackage[usenames]{color}
\usepackage{amsmath,amsthm}
\usepackage{enumerate}

\usepackage{tikz}
\usetikzlibrary{arrows}
\usepackage{mathdots}
\usepackage{mathtools}

\usepackage{stmaryrd}
\usepackage{dsfont}
\usepackage{algorithm}
\usepackage{algorithmicx}
\usepackage{algpseudocode}

\usepackage{todonotes}
\DeclareMathOperator{\Expand}{Expand}
\DeclareMathOperator{\ResApprox}{ResApprox}
\DeclareMathOperator{\STSolve}{STSolve}
\DeclareMathOperator{\Id}{Id}
\DeclareMathOperator{\ST}{ST}

\DeclareMathOperator{\rank}{rank}
\DeclareMathOperator{\supp}{supp}

\newcommand{\Restr}{\mathbf{R}}

\newcommand\diam{\mathop{\rm diam}}

\providecommand{\abs}[1]{\lvert#1\rvert}
\providecommand{\bigabs}[1]{\bigl\lvert#1\bigr\rvert}

\providecommand{\norm}[1]{\lVert#1\rVert}
\providecommand{\bignorm}[1]{\bigl\lVert#1\bigr\rVert}
\providecommand{\Bignorm}[1]{\Bigl\lVert#1\Bigr\rVert}
\providecommand{\biggnorm}[1]{\biggl\lVert#1\biggr\rVert}

\providecommand{\ceil}[1]{\lceil#1\rceil}

\newtheorem{theorem}{Theorem}
\newtheorem{lemma}[theorem]{Lemma}
\newtheorem{prop}[theorem]{Proposition}

\theoremstyle{definition}

\newtheorem{assumption}{Assumption}

\theoremstyle{remark}
\newtheorem{remark}[theorem]{Remark}

\newcommand{\cA}{{\mathcal{A}}}

\newcommand{\cM}{{\mathcal{M}}}

\newcommand{\cF}{{\mathcal{F}}}

\newcommand{\cV}{\mathcal{V}}
\newcommand{\cT}{\mathcal{T}}
\newcommand{\cI}{\mathcal{I}}
\newcommand{\cS}{\mathcal{S}}

\newcommand{\cG}{\mathcal{G}}

\newcommand{\bA}{\mathbf{A}}
\newcommand{\ba}{\mathbf{a}}
\newcommand{\bI}{\mathbf{I}}
\newcommand{\bu}{\mathbf{u}}
\newcommand{\bv}{\mathbf{v}}
\newcommand{\bw}{\mathbf{w}}
\newcommand{\bbf}{\mathbf{f}}
\newcommand{\bB}{\mathbf{B}}
\newcommand{\bM}{\mathbf{M}}
\newcommand{\bN}{\mathbf{N}}
\newcommand{\br}{\mathbf{r}}

\newcommand{\sdd}{\,\mathrm{d}}

\newcommand{\Chi}{\raise .3ex
\hbox{\large $\chi$}} 

\newcommand{\T}{{\mathbb{T}}}
\newcommand{\R}{\mathbb{R}}
\newcommand{\N}{\mathbb{N}}
\newcommand{\Z}{\mathbb{Z}}

\numberwithin{equation}{section}

\date{\today}

\title[Sparse and low-rank approximations: The best of both worlds]{Sparse and low-rank approximations of parametric elliptic PDEs: the best of both worlds}
\author{Markus Bachmayr$^1$}
\address{\rm $^1$ Institut f\"ur Geometrie und Praktische Mathematik, RWTH Aachen University, Templergraben 55, 52062 Aachen, Germany}
\email[Markus Bachmayr]{bachmayr@igpm.rwth-aachen.de}
\author{Huqing Yang$^1$}
\email[Huqing Yang]{yang@igpm.rwth-aachen.de}
\thanks{
The authors acknowledge funding by Deutsche Forschungsgemeinschaft (DFG, German Research Foundation) -- project number 442047500/SFB 1481 \emph{Sparsity and Singular Structures}. } 

\begin{document}

\begin{abstract}
A new approximation format for solutions of partial differential equations depending on infinitely many parameters is introduced. By combining low-rank tensor approximation in a selected subset of variables with a sparse polynomial expansion in the remaining parametric variables, it addresses in particular classes of elliptic problems where a direct polynomial expansion or a direct use of tensor decompositions in all variables may be inefficient, such as those arising from random diffusion coefficients with short correlation length. A convergent adaptive Galerkin method is proposed and analyzed that maintains quasi-optimal ranks of approximations. This is achieved by solving the Galerkin subproblems using a soft-thresholded Richardson iteration. At the same time, the method yields optimal convergence rates of spatial discretizations without coarsening. The results are illustrated by numerical tests.

\medskip
\noindent \emph{Keywords.}  parametric elliptic PDEs, sparse polynomial approximations, low-rank tensor representations, adaptivity
\smallskip

\noindent \emph{Mathematics Subject Classification.} {41A46, 41A63, 42C10, 65D99, 65J10, 65N12, 65N15}
\end{abstract}

\maketitle

\section{Introduction}\label{sec:intro}

Partial differential equations depending on a large or even infinite number of parameters are frequently encountered in applications. 
They can arise not only in a deterministic context in the exploration of high-dimensional parameter spaces, but also in problems of uncertainty quantification: when coefficients given by random fields are expanded as function series with scalar random coefficients, this also leads to formulations as deterministic parametric problems.

 In such problems, approximate solutions for many different parameter values, functionals of solutions, or their expectations may be required. 
A natural question is then how to construct efficient approximations of the \emph{solution map} that takes each parameter value $y$ from some admissible set $Y$ to the corresponding solution of the PDE problem.
In this article, we consider the parametric elliptic diffusion problem on a spatial domain $D \subset \R^d$, typically with $d \in \{1,2,3\}$,
\begin{equation}
    \label{eq:elliptic}
    \begin{aligned}x
		-\nabla \cdot(a(x,y) \nabla u(x,y)) & =f(x) \quad & & \text { in } D\times Y , \\
		u(x,y) & =0  \quad & & \text{ on } \partial D.
	\end{aligned}
\end{equation}
In the standard weak formulation of this problem, we assume $f \in L_2(D)$ to be given and for parameter-dependent diffusion coefficients $a(\cdot ,y)$ seek the corresponding solutions $u(\cdot, y) \in V = H^1_0(D)$.
We focus on the case where $a$ is parameterized in an affine manner in terms of uniformly distributed random variables $y = (y_1,y_2,\ldots)$, 
\begin{equation}
    \label{eq:affine}
    a\left(x, y_1, y_2, \ldots\right)=\bar{a}(x)+\sum_{j \in \cI} y_j \theta_j(x),
\end{equation}
with $\bar{a}, \theta_j \in L_{\infty}(D)$ and $y_j \in [-1,1]$ for each $j \in \cI$. Accordingly $Y = [-1,1]^{\cI}$, where $\cI$ may be of finite or of countably infinite cardinality. 
The latter case requires that $\norm{\theta_j}_{L_\infty} \to 0$ as $j\to \infty$, so that asymptotically for increasing $j$, the parametric variables $y_j$ have decreasing influence on $a$ and thus on the solution $u$.

The uniform ellipticity assumption
\begin{equation}
    \label{UEA}
    \bar{a}(x) - \sum\limits_{j \in \cI}|\theta_j(x)| \geq r \qquad \text{for almost every } x \in D,
\end{equation}
with a constant $r>0$ independent of $x$, ensures that $a(\cdot,y)\geq r > 0$ and hence the well-posedness of the problem \eqref{eq:elliptic} in $V = H_0^1(D)$ 
for all $y \in Y$. 
In the above setting, the mapping $y \mapsto u(\cdot, y) \in V$ has a high degree of regularity with respect to the parametric variables $(y_i)_{i \in \cI}$: it can be shown (see, e.g., \cite{CohenDeVore:15}) to have a holomorphic extension to a subset of $\mathbb{C}^\cI$ containing $Y$. Thus spectral methods based on multivariate polynomials on $Y$ are a natural choice for approximating $u$. For generating such polynomial approximations, there exists by now a wide variety of methods, where each favors a particular choice of basis of a space spanned by polynomials, such as Taylor expansions, sparse polynomial interpolation (also known as stochastic collocation), or orthonormal polynomials combined with variational (so-called \emph{stochastic Galerkin}) formulations. For an overview, we refer to \cite{SchwabGittelson:11,CDacta:15}.

The methods mentioned above are all based on a sparse selection of multivariate polynomial degrees.
However, other types of efficient representations in terms of a tractable number of coefficients can also be of interest.
One instance of such representations is nonlinear parameterization of the polynomial coefficients in \emph{low-rank tensor formats}. These generalizations of low-rank matrix representations to higher-order tensors, such as tensor trains \cite{Oseledets:11} or hierarchical tensors \cite{HackbuschKuehn:09}, provide structured representations of full arrays of coefficients corresponding to Cartesian product index sets. For general aspects of such methods and an overview of their applications, see, for example, \cite{Hackbusch:14,BachmayrSchneiderUschmajew:16,Hackbusch:19,Bachmayr:23}.

In approximating solutions to parameter-dependent elliptic problems, low-rank structures can play a role that is in a certain sense complementary to sparse expansions. 
Sparse polynomial approximations can give favorable asymptotic convergence rates in terms of the asymptotic decay of $\norm{\theta_j}_{L_\infty}$, and can thus take advantage of \emph{anisotropy} in the parameter dependence.
However, in cases where this decay sets in only after a certain plateau, so that a sizable number of the first parameters have comparable influence, sparsity-based methods are less effective. 

Conversely, low-rank approximations are observed to typically work well in cases of \emph{i\-so\-tro\-pic} parameter dependence, that is, with finitely many parameters of similar influence. This has been observed, for example, in \cite{BachmayrCohen:17}. 
However, in cases with an anisotropic dependence on \emph{infinitely} many parameters, they are typically comparably less efficient due to the substantial computational overhead of handling low-rank representations. 
In particular, when a low-rank separation is used between spatial and parametric variables, in the anisotropic case, the required approximation ranks can be too large to yield a gain in comparison to a direct sparse polynomial approximation of the solution \cite{BachmayrCohenDahmen:18}.

In problems of practical relevance, both isotropic and anisotropic influence of parameters can be present at the same time.
In other words, in problems with infinitely many parameters and decay $\norm{\theta_j}_{L_\infty} \to 0$, this decay often has a pronounced \emph{pre-asymptotic} regime with essentially isotropic parameter dependence, that is, for some $J$ and for $1 \leq j \leq J$, the values, $\norm{\theta_j}_{L_\infty}$ are of similar size so that the first $J$ parameters have comparably strong influence. This is the case, for example, in problems where $a$ represents a random field with \emph{short correlation length}, so that realizations exhibit strong oscillations within the domain $D$. Although the asymptotic convergence rate of sparse polynomial approximations then still only depends on the asymptotic decay of the $\norm{\theta_j}_{L_\infty}$, the total number of coefficients may generally depend exponentially on $J$, and thus exhibit a curse of dimension with respect to the number $J$ of leading parametric variables. 

The central question addressed in this paper is thus how low-rank tensor representations can be used to avoid exponential scaling in $J$, while at the same time maintaining the more favorable asymptotic computational costs of sparse polynomial approximations for the remaining parameters with $j > J$. The two main contributions of this work are the following:
\begin{enumerate}[(i)]
\item We propose a new form of low-rank tensor representation for problems with infinitely many parameters that preserves some key advantages of sparse polynomial approximations.
\item For low-rank approximations using this new structure, we analyze an adaptive low-rank solver with convergence and quasi-optimality guarantees for the generated ranks and the discretizations in the tensor modes.
\end{enumerate}

\subsection{A new low-rank representation}

We introduce a tensor representation for problems with infinitely many parameters based on a separate treatment of $(y_1,\ldots,y_J)$, with $J$ as above, which we call the \emph{dominant parameters,} and of the remaining parameters $(y_{J+1}, y_{J+2}, \ldots)$, which we call the \emph{tail parameters.} Here we take into account two observations: first, the use of low-rank representations in the dominant parameters can mitigate the curse of dimensionality with respect to $J$; second, we should not use a low-rank separation between spatial variables and tail parameters, since this would require large approximation ranks. This leads us to low-rank approximations of the basic form
\begin{equation}\label{eq:lowrankbasic}
  u(x, y_1, y_2,\ldots) \approx \sum_{k_1,k_2,\ldots,k_J} u^{(0)}_{k_1}(x, y_{J+1}, y_{J+2}, \ldots) \,u^{(1)}_{k_1,k_2}(y_1) \,u^{(2)}_{k_2,k_3}(y_2)  \cdots u^{(J)}_{k_J}(y_J) \,.
\end{equation} 
Here the summation is over $k_j = 1,\ldots, r_j$ with rank parameters $r_j \in \N_0$ for $j =1,\ldots, J$. While \eqref{eq:lowrankbasic} corresponds to the \emph{tensor train} format, we will also consider slightly more general structures in the \emph{hierarchical tensor format}. The main point, however, is that we separate the dominant parametric variables in a high-order low-rank tensor format, whereas the spatial variable $x$ and the tail parameters are kept in a single tensor mode. 

The univariate functions $u^{(j)}_{k_j,k_{j+1}}$ for $j=1,\ldots,J-1$ and $u^{(J)}_{k_J}$ can be approximated in a straightforward manner by polynomials.
However, the functions $u^{(0)}_{k_1}$ for $k_1 =1, \ldots, r_1$ still depend on countably many variables. Since they depend anisotropically on the parameters, they can be treated efficiently by sparse polynomial approximations. In our present case, for each $k_1$, these approximations take the form
\begin{equation}\label{eq:sparseinlowrank}
   u^{(0)}_{k_1} (x, y_{J+1}, y_{J+2}, \ldots)  =  \sum_{\nu \in \Lambda} w_{k_1,\nu}(x)  \prod_{j > J} L_{\nu_j} (y_j) ,
\end{equation}
where $\Lambda$ is a finite subset containing multi-indices $\nu = (\nu_{J+1}, \nu_{J+2}, \ldots)$ with $\nu_j \in \mathbb{N}_0$ for $j > J$ and $\nu_j = 0$ for all but finitely many $j$.
Here $L_{\nu_j}$ is the Legendre polynomial of degree $\nu_j$ and the spatial coefficient functions $w_{k_1, \nu}$ are piecewise polynomials (such as finite element or wavelet approximations). The best possible asymptotic convergence rate of expansions as in \eqref{eq:sparseinlowrank}, including also spatial discretizations, can be achieved under certain additional structural conditions on the functions $\theta_j$, see \cite{BachmayrCohenDungSchwab:17}. To this end, it is important that an independent finite element mesh (or subset of wavelet basis elements) can be used for each $\nu$ to represent $w_{k_1,\nu}$. 

Here an additional issue comes into play: low-rank tensor arithmetic requires repeated linear combinations of $u^{(0)}_{k_1}$ for different $k_1$, for example for orthogonalization. While the finite element meshes thus may vary depending on $\nu$, both these meshes and the subset $\Lambda$ of polynomial degrees need to be the same for all $k_1$ (since otherwise they will be merged during the computation). Altogether, these requirements are accommodated by the representations combining \eqref{eq:lowrankbasic} with \eqref{eq:sparseinlowrank} that we use in what follows.

\subsection{Stochastic Galerkin discretization}

Various approaches to numerical constructions of combined low-rank and sparse approximations as described above are possible. Here we focus on stochastic Galerkin methods, which are based on variational formulations on the full spatial-parametric domain $D \times Y$, since these are most suitable for adaptive refinement of discretizations using residual information.

As a consequence of \eqref{UEA}, the parameter-dependent solution $u$ can be regarded as an element of 
\[   \mathcal{V} = L_2(Y, V; \sigma)  \cong V \otimes L_2(Y;\sigma) ,  \] 
where $\sigma$ is a suitable probability measure on $Y$. This enables expansions of $u$ in terms of orthonormal bases of $L_2(Y;\sigma)$, with coefficients in $V$.
Our main requirement on $\sigma$ is that it can be written as a product of measures on $[-1,1]$, which enables the use of \emph{product} orthonormal bases of $L_2(Y;\sigma)$. For simplicity of presentation, we take $\sigma$ to be the uniform measure in what follows, which amounts to assuming independent, identically distributed $y_j \sim \mathcal{U}(-1,1)$ for each $j \in \N$. Note that both the choice of the interval $[-1,1]$ and of uniform distributions is made for convenience of presentation: other distributions on general bounded intervals could be treated in the same manner, for example beta distributions associated to Jacobi polynomials \cite{BachmayrCohenDungSchwab:17}.
With the $L_2(Y;\sigma)$-orthonormal product Legendre polynomials 
\[
   L_{\nu}(y) = \prod\limits_{i \in \cI}L_{\nu_i}(y_i), 
   \quad
    \nu \in \mathcal{F} = \bigl\{\nu \in \mathbb{N}_0^{\infty}: \nu_i 
   \neq 0 \text{ for finitely many } i \bigr\} ,
\]
we obtain the orthonormal basis expansion
\begin{equation}
    \label{sparse_approximation}
    u(x,y) = \sum_{\nu \in \mathcal{F}}u_{\nu}(x)L_{\nu}(y), \quad u_{\nu}(x) = \int_{Y}u(x,y)L_{\nu}(y)\sdd\sigma(y) \in V \,.
\end{equation}

The solution $u$ of \eqref{eq:elliptic}, as defined pointwise for each given $y \in Y$, can also be characterized by the spatial-parametric
 variational formulation: find $u \in \mathcal{V}$ such that 
\begin{equation}
    \label{eq:elliptic_vf}
    \int_{Y}\int_{D} a(x,y) \nabla u(x,y) \cdot \nabla v(x,y) \sdd x \sdd \sigma(y) = \int_{Y} \int_D f(x) \, v(x ,y) \sdd x \sdd \sigma(y)
\end{equation}
for all $v \in \mathcal{V}$. Stochastic Galerkin discretizations are obtained by replacing $\mathcal{V}$ by a finite-dimensional subspace.
In what follows, our objective is to use adaptively refined stochastic Galerkin discretizations to find approximations of the exact solution $u$ given by \eqref{eq:elliptic_vf} in a low-rank form as in \eqref{eq:lowrankbasic} with sparse expansion \eqref{eq:sparseinlowrank} in the first tensor mode.
Altogether, this requires identifying the discretizations used in each tensor mode as well as the appropriate rank parameters for the tensor representations, while at the same time computing the lower-order component tensors of the low-rank approximation.

\subsection{Relation to existing works}\label{sec:previouswork}

What distinguishes the approach of this paper from previous works on low-rank tensor approximations for parametric PDEs is that we do not separate the spatial variable from the tail parameters. In contrast, a standard approach is to introduce a low-rank separation between spatial and all parametric variables, which amounts to considering approximations of the form
\begin{equation}\label{eq:lrstandardform}
   u(x, y_1, y_2, \ldots) \approx
     \sum_{k=1}^r u^{\mathrm{(x)}}_k(x)  \,u^{\mathrm{(y)}}_k (y_1,y_2,\ldots),
\end{equation}
where depending on the particular representation format, the $u^{\mathrm{(y)}}_k$ can be decomposed further into tensors of lower order (such as in a tensor train or hierarchical tensor representation).

Approximations of the form \eqref{eq:lrstandardform} are also the basis of projection-based model order reduction methods. One example is the reduced basis method, where functions $u_k^{\mathrm{(x)}}$ for $k = 1, \ldots, r$ are found by point evaluations for fixed values of the parameters $y = (y_1, y_2, \ldots)$. The corresponding functions $u_k^{\mathrm{(y)}}$ of the parameters are then given implicitly by Galerkin projection.
In certain cases with finitely many parameters, best approximations of this form (using uniform or $L_2$-norms on the parameter domain) are known to converge exponentially with respect to $r$; see, for example, \cite{BachmayrCohen:17}. 

In contrast, in problems with infinitely many parameters, as observed in \cite{BachmayrCohenDahmen:18}, the best approximation errors typically decay only algebraically with respect to $r$.
This leads to a first limitation of any low-rank tensor approximation of the form \eqref{eq:lrstandardform} that separates all parametric variables from the spatial variables: the efficiency of the approximation is restricted by the value of $r$ required for the desired accuracy (and hence, by the number of required spatial basis functions $u^{\mathrm{(x)}}_k$). 

A second, slightly more subtle limitation concerns restrictions on spatial discretizations. Numerical methods that operate on approximations of the form \eqref{eq:lrstandardform} require repeated orthogonalizations of intermediate results for the functions $u^{\mathrm{(x)}}_k$. This means that there is no benefit in choosing different spatial discretizations for these functions, since these would be merged in the linear combinations required for orthogonalizations.  However, as the results in \cite[Sec.~7]{BachmayrCohenDungSchwab:17} show, using an independent spatial discretization for each term in an expansion can lead to improved rates of space-parameter approximation. 
The practical restriction to a single spatial mesh in \eqref{eq:lrstandardform} may thus limit the attainable convergence rate with respect to the total number of coefficients in the case of infinitely many parameters.

We address both limitations by, instead of separating spatial and all parametric variables as in a low-rank approximation \eqref{eq:lrstandardform}, using a low-rank separation only of spatial and \emph{dominant} parametric variables as in \eqref{eq:lowrankbasic} and treating the \emph{tail} parametric variables (which, together with the spatial variables, are in the first tensor mode) by a sparse polynomial expansion as in \eqref{eq:sparseinlowrank}. 

Our main focus in this work is on the convergence and complexity analysis of adaptive methods for generating combined sparse and low-rank approximations of this particular kind. 
Adaptive schemes based on stochastic Galerkin discretizations for sparse polynomial approximations \eqref{sparse_approximation} have been considered in numerous works, for example in \cite{Gittelson:13,EGSZ:14,BachmayrCohenDahmen:18,CPB:19,BPRR19,BachmayrVoulis:22,BEEV:24}.
Methods for low-rank approximation have been proposed in the case of finitely many parameters in \cite{KhoromskijSchwab11,Matthies:12,KressnerTobler,DKLM:15}, and for infinitely many parameters in \cite{EPS:17,BachmayrCohenDahmen:18,EMPS:20}, in each case using the approach \eqref{eq:lrstandardform} with separation of spatial and parametric variables.

\subsection{A novel adaptive low-rank solver}

The method that we consider here differs from known low-rank solvers not only in that it is based on the new approximation format \eqref{eq:lowrankbasic} and \eqref{eq:sparseinlowrank}, but also in the basic construction of the adaptive scheme, in particular concerning the interaction between low-rank solver and adaptive discretization refinement.

The method that we consider here is based on iterative refinements of Galerkin discretizations. Our main result is that we show quasi-optimality both of ranks in the tensor representations and of discretization sizes in the sparse representations in the tensor modes. In other words, we estimate these quantities by fixed multiples of the ranks or discretization sizes of best approximations of comparable accuracy.
Results of a similar kind have been obtained before for adaptive low-rank methods of a different type, for general elliptic operator equations in  \cite{BachmayrDahmen:14} and for parametric problems in \cite{BachmayrCohenDahmen:18}. However, the new adaptive scheme that is tailored to our particular approximation format for parametric problems yields two major improvements over these previous results.

The first improvement concerns a tighter control of approximations ranks. Rather than on intermittent recompression steps as in the iterative scheme used in \cite{BachmayrDahmen:14,BachmayrCohenDahmen:18}, where substantial rank growth in the steps between recompressions is not ruled out, 
we use a low-rank solver based on tensor soft thresholding as in \cite{BachmayrSchneider:17}. In this work, we combine this technique with adaptive discretization refinement. This leads to substantially improved estimates for the largest ranks of intermediate quantities that are produced by the iteration.

The second improvement is in ensuring optimality of discretizations in the tensor modes: In the approximations of the form \eqref{eq:lowrankbasic}, \eqref{eq:sparseinlowrank}, for the functions $u^{(0)}_{k_1}$, $u^{(1)}_{k_1,k_2}$, \ldots, $u^{(J)}_{k_J}$, we adaptively select suitable discretizations. Doing this efficiently is particularly challenging for the $u^{(0)}_{k_1}$ as in \eqref{eq:sparseinlowrank}, which generally are functions of countably many variables. Similarly as for the tensor ranks, the number of degrees of freedom used for the discretization in each tensor mode should be comparable to the minimally required number for the achieved discretization error.
In the method introduced in \cite{BachmayrDahmen:14} and applied to parametric problems in \cite{BachmayrCohenDahmen:18}, this is achieved by eliminating degrees of freedom in coarsening steps after a certain number of iterations. For standard adaptive Galerkin methods based on direct basis expansions, it was shown in \cite{GantumurHarbrechtStevenson:07} that such coarsening is not needed when the refinement is done sufficiently conservatively.
However, a direct adaptation of this technique to low-rank tensor case as in \cite{BachmayrDahmen:14} cannot be shown to be optimal in the higher-dimensional case, as explained in detail in \cite{AliUrban:20}. In the present work, we exploit the particular structure of the parametric problems and of the adaptive low-rank approximations for the components in \eqref{eq:lowrankbasic}, \eqref{eq:sparseinlowrank} to achieve optimality without coarsening.

\subsection{Outline}

In Section \ref{sec:sparselowrank}, we give a detailed description of our particular format for sparse and low-rank representation of parametric solutions, as well as of operators in stochastic Galerkin discretizations. Section \ref{sec:softthresh} is concerned with the construction of a solver for Galerkin discretizations in low-rank format. To this end, we use a method based on soft thresholding of hierarchical tensor representations that yields quasi-optimal ranks. In Section \ref{sec:adaptgalerkin}, we consider how this can be combined with discretization refinement in each tensor mode without intermittent coarsening steps. In Section \ref{sec:numtests}, we show results of some numerical tests, where we compare in particular to standard tensor representations as in \eqref{eq:lrstandardform}.

\section{Sparse and low-rank approximations for parametric problems}\label{sec:sparselowrank}

We now turn to a more detailed description of the new approximation format for problems with infinitely many parameters given by \eqref{eq:lowrankbasic} and \eqref{eq:sparseinlowrank}. We begin by recapitulating the relevant basics on sparse expansions and low-rank approximations. Moreover, we also consider corresponding representations of parametric differential operators, which play a central role in the methods considered in the following sections.

\subsection{Basis expansion}

In what follows, our minimal assumption on the parametrized coefficient $a$ is the uniform ellipticity condition \eqref{UEA}, which ensures that $u$ is an element of $\cV = V \otimes L_2(Y;\sigma)$.
One natural way to construct a polynomial approximation is then to restrict the summation over $\nu$ in \eqref{sparse_approximation} to a finite subset  
$\Lambda \subset \mathcal{F}$,
\begin{equation}\label{eq:semidiscr}
    u(x,y) \approx \sum_{\nu \in \Lambda}u_{\nu}(x)L_{\nu}(y).
\end{equation}
In particular, choosing $\Lambda$ to contain the indices corresponding to the $n$ largest $\|u_{\nu}\|_V$ yields a semidiscrete best 
$n$-term approximation. 

Moreover, a suitable spatial approximation is required for the Legendre coefficients in \eqref{sparse_approximation}.
In the literature, common choices are finite elements \cite{GhanemSpanos:91,GhanemSpanos:97,EGSZ:14,BPRR19,BEEV:24,BEV} or wavelet approximations \cite{Gittelson:14,BachmayrCohenDahmen:18,BachmayrVoulis:22}.
To obtain a more transparent construction of the adaptive scheme, as a first step we focus here on discretizations by subsets of a wavelet Riesz basis of $V$. 
In what follows, we assume $\Psi = \{\psi_\lambda\}_{\lambda \in \cS}$, with a countable index set $\cS$, to be a Riesz basis of $V$. This means that there exist $c_\Psi, C_\Psi$ such that for all $\bv \in \ell_2(\cS)$,
\begin{equation}
    \label{Riesz_basis}
    c_{\Psi} \norm{ \bv }_{\ell_2} \leq \Bignorm{\sum_{\lambda \in \cS} \bv_{\lambda}\psi_{\lambda} }_V \leq C_{\Psi} \norm{ \bv }_{\ell_2} \,.
\end{equation}
Since $\{ L_\nu \}_{\nu \in \cF}$ is an orthonormal basis of $L_2(Y;\sigma)$, the product basis $\{\psi_{\lambda}\otimes L_{\nu}\}_{\nu \in \mathcal{F}, \lambda \in \cS}$ is a also a Riesz basis of $\cV$ with the same constants as in \eqref{Riesz_basis}. 
By the expansion
\begin{equation}
    \label{sparse_approximation_sequence}
    u = \sum\limits_{\nu \in \mathcal{F}} \sum_{ \lambda \in \cS} \mathbf{u}_{\nu,\lambda}\psi_{\lambda} \otimes L_{\nu} 
\end{equation}
converging in $\cV$ and the norm equivalence \eqref{Riesz_basis}, the problem of approximating $u$ in $\cV$ reduces to approximating the coefficient sequence $\mathbf{u}$ in $\ell_2(\cF\times\cS)$. 
One can again obtain best $N$-term approximations $\bu_N$ to $\bu$ by restricting the summation in \eqref{sparse_approximation_sequence} to a finite set $\Lambda \subset \mathcal{F} \times \cS$ with indices corresponding 
to the $N$ largest $\abs{\mathbf{u}_{\nu,\lambda}}$. By \eqref{Riesz_basis}, these are quasi-best approximations of $u$ in $\cV$. 

Plugging expansions of $u, v \in \cV$ as in \eqref{sparse_approximation_sequence} into \eqref{eq:elliptic_vf} leads to the operator equation
\begin{equation}
    \label{eq:sequence_form}
     \mathbf A \bu \coloneqq \sum\limits_{i\in \{0\} \cup \cI} (\bM_i \otimes \bA_i) \mathbf{u} = \bbf
\end{equation}
on the sequence space $\ell_2(\cF\times \cS)$, which is equivalent to \eqref{eq:elliptic_vf}.
Here $\bM_i: \ell_2(\mathcal{F}) \to \ell_2(\mathcal{F})$ are defined by the action of the bi-infinite matrices
\begin{equation}
    \label{def_Mi}
    \begin{aligned}
    \bM_0 & =\left(\int_Y L_\nu(y) L_{\nu^{\prime}}(y) \sdd \sigma(y)\right)_{\nu, \nu^{\prime} \in \mathcal{F}}, \\
    \bM_i & =\left(\int_Y y_i L_\nu(y) L_{\nu^{\prime}}(y) \sdd \sigma(y)\right)_{\nu, \nu^{\prime} \in \mathcal{F}}, \quad i \in \cI,
    \end{aligned}
\end{equation}
$\bA_i: \ell_2(\cS) \to \ell_2(\cS)$ are given by
\begin{equation}
    \label{def_Ai}
    \begin{aligned}
    \bA_0 & = \left(\int_D \bar{a}(x) \nabla\psi_{\lambda^{\prime}}(x) \cdot \nabla\psi_{\lambda}(x) \sdd x\right)_{\lambda,\lambda^{\prime}\in \cS}, \\
    \bA_i & = \left(\int_D \theta_i(x) \nabla\psi_{\lambda^{\prime}}(x) \cdot \nabla\psi_{\lambda}(x)\sdd x\right)_{\lambda,\lambda^{\prime}\in \cS}, \quad i \in \cI,
    \end{aligned}    
\end{equation}
and furthermore, 
\[
    \bbf = \left(\int_D f \psi_{\lambda}\sdd x \int_Y L_{\nu}(y) \sdd \sigma(y)\right)_{(\nu,\lambda)\in \mathcal{F}\times \cS}.
\]
As a consequence of the three-term recursion relation for orthonormal Legendre polynomials on $[-1,1]$,
\begin{equation}\label{eq:threeterm}
  yL_i(y) = \sqrt{\beta_{i+1}}L_{i+1}(y) + \sqrt{\beta_i}L_{i-1}(y), \qquad L_0 \equiv 1,\; L_{-1} \equiv 0,
\end{equation}
with $\beta_i = (4-i^{-2})^{-1}$, the operators $\bM_i$ are bidiagonal,
\[
\begin{aligned}
& \bM_0=\left(\delta_{\nu, \nu^{\prime}}\right)_{\nu, \nu^{\prime} \in \mathcal{F}}, \\
& \bM_i=\left(\sqrt{\beta_{\nu_i+1}} \delta_{\nu+\mathrm e_i, \nu^{\prime}}+\sqrt{\beta_{\nu_i}} \delta_{\nu-\mathrm e_i, \nu^{\prime}}\right)_{\nu, \nu^{\prime} \in \mathcal{F}}, i \in \cI,
\end{aligned}
\]
with the Kronecker sequences $\mathrm e_i$ defined by $(\mathrm e_i)_j = \delta_{ij}$, $j \in \mathbb{N}$.

\subsection{Convergence of sparse best approximations}

The convergence of semi-discrete best $n$-term approximations by selection of $\Lambda_n \subset \cF$ with $\#\Lambda_n = n$ in \eqref{eq:semidiscr} is by now a well-studied problem. The convergence of such approximations in $L_2(Y,V,\sigma)$ can be quantified in terms of the $\ell_p$-summability of the coefficient norms $\|u_{\nu}\|_V$ for $\nu \in \cF$: If $(\|u_{\nu}\|_V)_{\nu \in \mathcal{F}} \in \ell_p(\cF)$ for a $p \in (0,2)$, then by Stechkin's lemma,
\[
    \biggnorm{ u - \sum_{\nu \in \Lambda_n} u_\nu L_\nu }_{L_2(Y,V,\sigma)} \leq C (n+1)^{-(\frac{1}{p}-\frac{1}{2})}, \quad C=\bignorm{(\norm{u_{\nu}}_V)_{\nu \in \mathcal{F}} }_{\ell_p}.
\]
A first result, obtained in \cite{CohenDeVoreSchwab:10,CohenDeVoreSchwab:11}, is that if \eqref{UEA} holds and $(\|\theta_j\|_{L_{\infty}})_{j\geq 1} \in \ell_p(\mathbb{N})$ with a $p \in (0,1)$, then $(\|u_{\nu}\|_V)_{\nu \in \mathcal{F}} \in \ell_p(\cF)$.
This is generally sharp, in particular for $\theta_j$ that have global supports in $D$. 

However, when the $\theta_j$ have localized supports, improved results are possible. In particular, this is the case when these functions have wavelet-like multilevel structure with a level parameter $\ell \in \N_0$ and a translation parameter $k \in \mathcal{I}_\ell \subset \Z$. In the following set of assumptions, it is convenient to write these functions as $\theta_{\ell, k}$. With a bijection between such $(\ell, k)$ and $j \in \N$ defined by $j  = j(\ell, k) = \sum_{n = 0}^{\ell-1}\#\mathcal{I}_n + k$, and with $\theta_j = \theta_{\ell,k}$ for all $\ell$ and $k$, we again obtain an enumeration $(\theta_j)_{j \in \N}$ ordered by increasing level parameter $\ell$. 

\begin{assumption}
    \label{multilevel_structure}
    We assume $\theta_{\ell, k} \in L_{\infty}(D)$ for $\ell \geq 0$, $k \in \mathcal{I}_\ell$ such that the following conditions hold:
    \begin{enumerate}[{\rm (i)}]
        \item $ \diam\supp(\theta_{\ell, k}) \eqsim 2^{-\ell}$ for all $\ell$ and $k$,
        \item $ \# \mathcal{I}_\ell \lesssim 2^{d\ell}$ for all $\ell$,
        \item for some $\alpha >0$, one has $\|\theta_{\ell, k}\|_{L_{\infty}} \lesssim 2^{-\alpha \ell}$.
    \end{enumerate}
\end{assumption}

Multilevel localized structure of $(\theta_j)_{j \in \N}$ as in Assumption \ref{multilevel_structure} is not a major restriction; for example, for all Mat\'ern covariances, there exist  decompositions with such wavelet-like localization \cite{BCM:18}.
For such $\theta_j$, as a particular case of a result in \cite{BachmayrCohenMigliorati:17}, for any $\alpha > 0$ we have $(\|u_{\nu}\|_V)_{\nu \in \mathcal{F}} \in \ell_p(\mathcal{F})$ for any $p > (\frac{\alpha}d + \frac12)^{-1}$, which yields
\[    \biggnorm{ u - \sum_{\nu \in \Lambda_n} u_\nu L_\nu }_{L_2(Y,V,\sigma)} \leq C (n+1)^{-s} \quad \text{ for any $\displaystyle s < \frac\alpha{d}$.} 
\]

For a computationally implementable approximation, we replace each $u_{\nu}$ by an approximation from a finite-dimensional subspace $V_{\nu} \subset V$, typically generated by some type of adaptive refinement of meshes or by selection from some fixed reference basis in $V$. We thus seek a fully discrete approximation of $u$ from a space
\[
    \mathcal{V}_{N} = \left\{\sum\limits_{\nu \in \Lambda_n}u_{\nu}L_{\nu}: u_{\nu}\in V_{\nu}, \nu \in \Lambda_{\nu}\right\} \subset L_2(Y,V,\sigma),
\]
with error estimates in terms of total number of degrees of freedom $N = \sum_{\nu \in \Lambda_n}\dim(V_{\nu})$. Here $s$ depends not only on the $\ell_p$-summability of $(\|u_{\nu}\|_V)_{\nu \in \cF}$,
but also on higher-order norms of the coefficients and on the order of approximation of the spaces $V_{\nu}$.
As shown in \cite[Cor.~8.8]{BachmayrCohenDungSchwab:17}, under Assumption \ref{multilevel_structure} with $\theta_j \in W^{1,\infty}(D)$ and additional conditions on $f$ and $D$ (see \cite[Sec.~7]{BachmayrCohenDungSchwab:17}), for $\alpha \in (0,1]$ we have
\begin{equation}
    \label{fully_discrete_approximation_rate}
    \min _{v \in \mathcal{V}_{N}}\|u-v\|_{\mathcal{V}} \leq C N^{-s} \quad \text { for any } s< \begin{cases}\displaystyle \frac{\alpha}{d}, & d \geq 2, \\[12pt] \displaystyle\frac{2}{3} \alpha, & d=1. \end{cases}
\end{equation}

In particular, for $d \geq 2$, there is no loss in the asymptotic convergence rate in passing from the semidiscrete to the fully discrete approximation. This can be interpreted as a high degree of sparsity in the spatial discretizations as a consequence of the localized supports of the functions $\theta_j$. The required separate adaptive choice of the discretization subspace $V_\nu$ for each $\nu$ can be achieved computationally at near optimal costs using wavelet or finite element discretizations, as shown in \cite{BachmayrVoulis:22,BEEV:24,BEV}. 
Note that multilevel structure of the random field expansion as in Assumption \ref{multilevel_structure} is crucial for obtaining \eqref{fully_discrete_approximation_rate}. In particular, the requisite spatial sparsity cannot generally be expected for Karhunen-Lo\'eve type expansions.

However, this favorable result depends on the decay of $\norm{ \theta_j }_{L_\infty}$; in cases as outlined in the previous section, with a preasymptotic range $j=1,\ldots, J$ of dominant parameters followed by tail parameters with the expected decay, the constant $C$ in \eqref{fully_discrete_approximation_rate} can be expected to be large and may grow exponentially on $J$.
Low-rank tensor approximations can be effective in problems with many dominant parameters (as observed, for example, in \cite{BachmayrCohen:17}). We thus combine the above sparse expansions with a particular form of the hierarchical tensor format.

\subsection{Hierarchical tensors}

The starting point for low-rank approximation of $u$ is to regard the coefficient sequence $\bu = ( \bu_{\nu,\lambda} )_{\nu \in \cF , \lambda \in \cS}$ as a tensor on an appropriate product index set. For example, considering $\bu$ as a bi-infinite matrix with row indices $\nu$ and column indices $\lambda$, which induces a Hilbert-Schmidt operator from $\ell_2(\cS)$ to $\ell_2(\cF)$, we obtain the singular value (or Hilbert-Schmidt) decomposition
\[
   \bu_{\nu, \lambda} = \sum_{k=1}^\infty \sigma_k \bu_{k, \lambda}^{(\mathrm{x})} \bu_{k, \nu}^{(\mathrm{y})} , \qquad \lambda \in \cS, \; \nu \in \cF,
\]
with non-negative and decreasing singular values $(\sigma_k)_{k \in \N}$ and orthonormal systems $(\bu_{k}^{\mathrm{x}})_{k \in \N}$ and $(\bu_{k}^{\mathrm{y}})_{k \in \N}$ in $\ell_2(\cS)$ and $\ell_2(\cF)$, respectively, which are called left and right singular vectors. 
Directly truncating the sum over $k$ yields a low-rank representation of $u$,
\begin{equation}\label{eq:lrstandarddiscrete}
u \approx \sum_{k = 1}^n \sigma_k \left(\sum\limits_{\lambda \in \cS} \mathbf{u}^{(\mathrm{x})}_{k,\lambda} \psi_{\lambda} \right)
    \otimes \left(\sum\limits_{\nu \in \cF} \mathbf{u}^{(\mathrm{y})}_{k,\nu} L_{\nu}\right).
    \end{equation}
It is a best approximation of $\bu$ in $\ell_2(\cF\times \cS)$ and thus a quasi-best approximation of $u$ in $\cV$ (up to constants depending on the spatial Riesz basis).   

As mentioned in Section \ref{sec:previouswork}, in problems with infinitely many parameters (that is, with unbounded $\cI$), the decay of the singular values will generally only be algebraic. Thus when taking also the costs of computing the singular vectors into account, one cannot necessarily expect any gain from low-rank approximations in this form for unbounded $\mathcal I$, as considered in detail in \cite{BachmayrCohenDahmen:18}.

Numerically realizable approximations also require discretizations in the spatial and parametric variables, which corresponds to restricting the summations in \eqref{eq:lrstandarddiscrete} to finite subsets of $\cS$ and $\cF$. 
For algorithmic reasons related to required orthogonalizations, these subsets cannot be chosen to depend on $k$, that is, we need to operate on a single spatial discretization. For this reason, such low-rank approximations cannot achieve the ideal fully adaptive asymptotic convergence rates for sparse approximations as in \eqref{fully_discrete_approximation_rate}, which require an individually adapted subset of $\cS$ for each value of $\nu \in \cF$, corresponding to separate spatial discretizations for each Legendre coefficient $u_\nu$. In what follows, we thus consider different tensor representations of $\bu$ that avoid a direct separation of the spatial from the parametric variables.

Here, we are interested in the case $\cI = \N$, where the first $J$ coefficients have a comparable influence on the solution and the rest have decreasing influence. Here a first indicator for the influence of the parameter $y_j$, $j \in \N$, is the size of $\norm{\theta_j}_{L_\infty}$; a systematic algebraic decrease of these quantities for $j > J$ leads to favorable convergence of sparse polynomial approximations in the parameters $y_{j}$, $j>J$.
We introduce a modified hierarchical tensor approximation
where the dominant $J$ parameters $y_1,\dots,y_J$ are assigned to separate modes, while the spatial variables $x$ and the remaining tail parameters $y_{J+1},y_{J+2},\dots$ are kept in a single tensor mode, which we refer to as the sparse expansion mode, and which can be handled by a standard sparse polynomial approximation in the tail parametric variables.

\subsubsection{Hierarchical tensor format}

We first briefly recall the definition and some basic properties of the hierarchical tensor format and how it leads in particular to low-rank representations of the form \eqref{eq:lowrankbasic}, \eqref{eq:sparseinlowrank}. With $J$ the number of dominant parametric variables, we consider $\bu$ as a tensor with the modes corresponding to these variables numbered as $1,\ldots, J$, and an additional first tensor mode, numbered $0$, comprising the variables $x$, $y_{J+1}, y_{J+2},\ldots$
For the set of all tensor modes we accordingly write $\hat{\cI}= \{0, 1,\dots,J\}$.

For convenience, we denote the index sets associated to tensor modes as follows: $\mathcal{G}_{0}= \cF \times \cS$, $\mathcal{G}_1= \N_0$, \ldots, $\mathcal{G}_J= \N_0$ and set $\cG= \cG_{0} \times\cdots \times \cG_J$, which can again be identified bijectively with $\cF\times \cS$. 
Grouping the variables according to these modes, as a first step we consider the approximation of  the solution map of \eqref{eq:elliptic} in \emph{Tucker tensor format}, which takes the form
\begin{equation}
    \label{HT-Sparse}
    u(x,y_1,\dots) \approx \sum_{k_0=1}^{r_0} v_{k_0}^{(0)}(x,y_{J+1},\dots)  \left( \sum_{k_1=1}^{r_1} \cdots \sum_{k_J=1}^{r_J}\ba_{k_0,k_1,\dots,k_J} \prod_{j=1}^J v_{k_j}^{(j)}(y_j)\right)
\end{equation}
with scalar coefficients $\ba_{k}$ with indices $k = (k_0, \ldots, k_J) \in \prod_{i=0}^{J} \{ 1,\ldots, r_i\}$.
Here $v_{k_0}^{(0)}$ are given by sparse polynomial expansions, with finite $\Lambda_{0} \subset \cG_0$, as
\begin{equation}
    \label{HT-Sparse-expansion}
    v_{k_0}^{(0)}(x,y_{J+1},\dots) = \sum\limits_{(\bar \nu,\lambda) \in \Lambda_{0}} \mathbf{U}_{(\bar \nu,\lambda),k_0}\psi_{\lambda}(x) \left(\prod\limits_{i = 1}^\infty L_{\bar \nu_i}(y_{J + i})\right)\,.
\end{equation}
For $j \in \{1,\dots,J\}$, the further tensor components $v_{k_j}^{(j)}(y_j)$ are discretized by univariate Legendre polynomials $\{L_k\}_{k\in \Lambda_j}$
with finite $\Lambda_{j} \subset \cG_j$ and coefficients $(\mathbf{U}^{(j)}_{\nu_j,k_j})_{\nu_j \in \Lambda_{j}}$. 

For a given error tolerance $\varepsilon>0$, we thus obtain a fully discrete approximation $u_{\varepsilon}$ of $u$, represented by its basis coefficients $\bu_{\varepsilon} \in \ell_2(\cG)$, where in view of the Riesz basis property \eqref{Riesz_basis}, we aim to achieve $\norm{\bu - \bu_\varepsilon} \leq \varepsilon$.
Written in terms of basis coefficients, \eqref{HT-Sparse} and \eqref{HT-Sparse-expansion} take the form
\begin{equation}
    \label{Tucker-Tensor-coefficients}
    \bu_{\varepsilon} = \sum_{k_0=1}^{r_0}\sum_{k_1=1}^{r_1}\cdots \sum_{k_J=1}^{r_J} \mathbf{a}_{k_0,k_1,\dots,k_J} \bigotimes_{j = 0}^J \mathbf{U}^{(j)}_{k_j},
\end{equation}
where $\mathbf{U}^{(0)}_{k_0} = (\mathbf{U}_{(\bar \nu,\lambda),k_0})_{(\bar\nu,\lambda) \in \Lambda_{0}}$ and $\mathbf{U}^{(j)}_{k_j} = (\mathbf{U}^{(j)}_{\nu_j,k_j})_{\nu_j \in \Lambda_{j}}$ for $j \in \{1,\dots,J\}$ are called the \emph{mode frames} of the representation.
The minimal $(r_0, r_1,\ldots,r_J)$ such that $\bu_{\varepsilon}$ can be represented in the form \eqref{Tucker-Tensor-coefficients} are the \emph{multilinear ranks} of $\bu_{\varepsilon}$.

In \eqref{Tucker-Tensor-coefficients}, if the high-order core tensor $(\mathbf{a}_{k_0,k_1,\ldots,k_J})_{k_0,k_1,\ldots,k_J}$ is represented directly by its entries, this is the classical Tucker tensor format. Generally, the number of degrees of freedom of the Tucker format increases exponentially with respect to the tensor order, which in the present case is $J+1$. We thus instead use the \emph{hierarchical} (or \emph{hierarchical Tucker}) tensor format \cite{HackbuschKuehn:09}, which amounts to a further decomposition of the core tensor $\mathbf{a}_{k_0,k_1,\ldots,k_J}$
in terms of lower-order tensors. 
In the presence of additional, more restrictive low-rank approximability properties, this format allows for a more favorable dependence of the resulting complexity on $J$.

For the construction of the hierarchical tensor format for the set of tensor modes $\mathcal{\hat I}$, following \cite{HackbuschKuehn:09}, let $\T \subset 2^{\hat{\cI}}$ be a binary dimension tree for tensor order $J+1$, that is, $\T$ is a set of non-empty subsets of $\hat{\cI}$ such that the following conditions hold:
\begin{enumerate}[{\rm(i)}]
    \item $\hat{\cI} \in \T$, and for each $i \in \{0, 1, \dots, J\}$ we have $\{i\} \in \T$.
    \item Each node $\alpha \in \T$ is either a leaf or there exist disjoint $\alpha_1, \alpha_2 \in \T$ such that $\alpha = \alpha_1 \cup \alpha_2$.
\end{enumerate}
In what follows, we focus on the linear dimension tree
\begin{equation}
    \label{linear_tree}
    \T = \bigl\{\hat{\cI}, \{0\}, \{1,\dots,J\}, \{1\}, \{2,\dots,J\}, \dots, \{J-1\}, \{J\} \bigr\}.
\end{equation}
A visualization of the dimension tree \eqref{linear_tree} is given in Figure \ref{fig:linear_tree}. While for notational simplicity, we focus on this particular structure in what follows, our further considerations apply to arbitrary binary dimension trees.
\begin{figure}[htbp]
    \hspace{-3cm}
    \Tree [.$\{0,1,2,\dots,J\}$
        [.$\{0\}$ ]
        [.$\{1,\dots,J\}$
            [.$\{1\}$ ]
            [.$\{2,\dots,J\}$
                [.$\{2\}$ ]
                [.$\cdots$
                    [.$\{J-1\}$ ]
                    [.$\{J\}$ ]
                ]
            ]
        ]
    ]
    \caption{The linear dimension tree as defined in \eqref{linear_tree}.}
    \label{fig:linear_tree}
\end{figure}

In order to represent $\bu \in \ell_2(\cG)$ in the hierarchical tensor format associated to $\T$, we consider matricizations $\cM_{\alpha}(\mathbf{u})$
for each node $\alpha \in \T$. With $\alpha^c= \hat{\cI} \backslash \alpha$, we define $\cM_{\alpha}(\mathbf{u})$ as a Hilbert-Schmidt operator
\[
    \cM_{\alpha}(\mathbf{u}): \ell_2(\times_{i \in \alpha^c}\mathcal{G}_i) \mapsto \ell_2(\times_{i \in \alpha}\mathcal{G}_i)
\]
induced by $\bu$ regarded as a bi-infinite matrix, where the modes in $\alpha$ are taken as row indices and those in $\alpha^c$ as column indices.
Note that $\rank(\cM_{\alpha}(\bu)) = \rank(\cM_{\alpha^c}(\bu))$. We define the set of \emph{effective edges} of the tensor format
\[
    \mathbb{E} = \bigl\{ \{ \alpha, \alpha^c\}: \alpha \in \T \backslash \hat{\cI} \bigr\}.
\]
The elements of $\mathbb{E}$ correspond to the matricizations resulting from $\T$ that are not the same up to transposition; in other words, these are the edges in $\T$, with the exception that the two edges below the root element $\hat{\cI}$ are a single effective edge. 
It is then easy to verify that for $E = \# \mathbb{E}$, we have $E = 2J - 1$. 

In what follows, we assume a fixed enumeration $(\{\alpha_t, \alpha_t^c\})_{t=1,\dots,E}$ of $\mathbb{E}$ such that $\alpha_t \in \T$ for $t = 1,\dots, E$. We abbreviate $\cM_t(\bu)= \cM_{\alpha_t}(\bu)$, denote the singular values of $\cM_t(\bu)$ by $\sigma_t(\bu)$, and write
\[
  \rank_t(\bu) = \rank \cM_t(\bu) =  \#\{k \in \N:\sigma_{t,k}\neq 0\} .
\] 
The hierarchical rank of $\bu$ corresponding to the hierarchical format with dimension tree $\T$ is then given by the tuple
\[
    \rank(\bu) = \bigl(\rank_t(\bu) \bigr)_{t=1,\dots,E} .
\]

In the case of the linear dimension tree \eqref{linear_tree}, we obtain an explicit recursive decomposition of $\mathbf{a}_{k_0,k_1,\dots,k_J}$ in \eqref{Tucker-Tensor-coefficients}. With
\[
    \tilde{r}_i = \rank \bigl( \mathcal{M}_{\{i,\dots,J\}}(\bu_{\varepsilon}) \bigr), \quad i = 1,\dots,J,
\]
we obtain (see, e.g., \cite[Sec.~2.5]{Bachmayr:23}) that $\ba_{k_0,k_1,\dots,k_J}$ can be decomposed in the form
\begin{equation}\label{eq:adecomp}
    \ba_{k_0,k_1,\dots,k_J} = \sum_{l_1=1}^{\tilde{r}_1} \mathbf{C}^{(0)}_{k_0,l_1} \sum_{l_2=1}^{\tilde{r}_2} \mathbf{C}^{(1)}_{l_1, k_1, l_2} \cdots 
    \sum_{l_{J}=1}^{\tilde{r}_{J}}  \mathbf{C}^{(J-1)}_{l_{J-1}, k_{J-1},l_J} \mathbf{C}^{(J)}_{l_J, k_{J}},
\end{equation}
with component tensors $\mathbf{C}^{(i)}$, $i = 1,\ldots, J$, of order at most three. Finally, we represent the approximate solution coefficients $\bu_{\varepsilon}$ in hierarchical tensor format with sparse expansions in the form
\begin{multline}
    \label{HT-Sparse-coefficients}
        (\bu_{\varepsilon})_{(\bar{\nu},\lambda),\nu_1,\dots,\nu_J}  = \sum_{k_0=1}^{r_0}\sum_{k_1=1}^{r_1}\cdots \sum_{k_J=1}^{r_J} \sum_{l_1=1}^{\tilde{r}_1} \cdots \sum_{l_J=1}^{\tilde{r}_J}  \mathbf{C}^{(0)}_{k_0,l_1}  \mathbf{C}^{(1)}_{l_1, k_1, l_2} \cdots  \\
     \cdots  \mathbf{C}^{(J)}_{l_J, k_{J}} \mathbf{U}^{(0)}_{(\bar{\nu},\lambda),k_0} \mathbf{U}^{(1)}_{\nu_1,k_1} \cdots \mathbf{U}^{(J)}_{\nu_J,k_J},
\end{multline}
supported on $\Lambda = \Lambda_0 \times \Lambda_1 \times \cdots \times \Lambda_J \subset \cG$. The exact solution $\bu$ can also be represented in this form 

Note that by inserting \eqref{eq:adecomp} into \eqref{HT-Sparse} and recombining components in the form
\[
  u^{(0)}_{l_1} =  \sum_{k_0 = 1}^{r_0} \mathbf{C}^{(0)}_{k_0,l_1} v^{(0)}_{k_0},  \quad u^{(1)}_{l_1,l_2} = \sum_{k_1=1}^{r_1} \mathbf{C}^{(1)}_{l_1, k_1, l_2} v^{(1)}_{k_1}, \; \ldots,
\]
we recover precisely the approximation format \eqref{eq:lowrankbasic} corresponding to a \emph{tensor train} representation. However, especially for the mode associated to the index set $\mathcal{G}_0$, it is advantageous for the computational complexity to keep the mode frames separated, and it may also be of interest to consider more general dimension trees than the linear ones that we consider here for simplicity.

Note that the type of tensor representation used here does not involve any explicit truncation of the parametric dimension: The sparse polynomial approximations corresponding to $\Lambda_0$ can depend on any of the tail parametric variables $y_{J+1}, y_{J+2}, \ldots$; this is in contrast to standard approaches to low-rank approximations to this problem, which usually start from a separation of spatial and parametric variables as in \eqref{eq:lrstandardform}. For example, in the methods in \cite{EPS:17,EMPS:20}, the parametric variables are restricted to a subset $y_1, \ldots, y_{J_\mathrm{trunc}}$ with sufficiently large $J_{\mathrm{trunc}}$, and each of these remaining parametric dimensions is assigned to a separate tensor mode in the hierarchical format.

\subsubsection{Stochastic Galerkin discretization}

We now cast the operator equation \eqref{eq:sequence_form}, 
\begin{equation}
    \label{operator-equation}
    \bA \bu = \bbf,
\end{equation}
which is just a basis representation of the original problem \eqref{eq:elliptic_vf},
in a form that is compatible with the above tensor approximation format. The solution $\bu \in \ell_2(\cG)$ is approximated in a hierarchical tensor format with a sparse polynomial expansion, and the operator thus needs to be written in a form that matches the structure of this representation. 

We observe that the operator $\bA$ can be written in the form
\begin{equation}
    \label{operator-TT-form}
    \bA = \bar{\bA} \otimes \bI_1 \otimes \cdots \otimes \bI_J + \bB_1 \otimes \bN_1 \otimes \bI_2 \otimes \cdots \otimes \bI_J + \cdots + \bB_J \otimes \bI_{1} \otimes \cdots \bI_{J-1} \otimes \bN_J,
\end{equation}
where $\bar\bA \colon \ell_2(\cG_{0}) \mapsto \ell_2(\cG_{0})$ and $\bB_j \colon \ell_2(\cG_{0}) \mapsto \ell_2(\cG_{0})$ for $j \in \{1,\dots,J\}$ are defined by
\begin{equation}
    \bar{\bA} = \bM_0 \otimes \bA_0 + \sum_{i = J+1}^\infty (\bM_i \otimes \bA_i), \quad \bB_{j} = \bM_0 \otimes \bA_j,
    \label{def_A_tensor}
\end{equation}
with $\bM_i$ and $\bA_i$ defined as \eqref{def_Mi} and \eqref{def_Ai}, respectively. In addition, $\bI_j$ is the identity operator on $\ell_2(\cG_j)$ and $\bN_j : \ell_2(\cG_{j}) \mapsto \ell_2(\cG_j)$ are given by
\begin{equation}
    \label{def_M_tensor}
    \bN_j =\left(\int_{-1}^{1} y_j L_{\nu_j}(y_j) L_{\nu_j^{\prime}}(y_j) \sdd \sigma(y_j)\right)_{\nu_j, \nu_j^{\prime} \in \cG_j}, \quad j \in \{1,\dots,J\}.
\end{equation}
With a decomposition of this form, \eqref{operator-TT-form} thus acts precisely on the tensor modes of approximations \eqref{HT-Sparse-coefficients}, with an operator representation rank bounded by $J+1$.

For any given error tolerance $\varepsilon$, we aim to derive an approximation $\bu_{\varepsilon}$ as above such that $\norm{\bu - \bu_\varepsilon}\leq \varepsilon$ in the form \eqref{HT-Sparse-coefficients} with quasi-optimal ranks and quasi-optimal discretization sizes $\sum_{j=0}^J \#\Lambda_j$. We will realize this by an adaptive Galerkin method with residual-based discretization refinement, where the Galerkin problem in each step is solved by an iterative scheme operating on the low-rank representations in hierarchical format introduced above.

\section{Soft-thresholding low-rank solver for parametric problems}\label{sec:softthresh}

\label{Sect:soft}
We briefly review the iterative method based on soft thresholding of hierarchical tensors introduced in \cite{BachmayrSchneider:17}. This approach combines convergent iterative schemes in low-rank format with an adaptive adjustment of thresholding parameters, such that the arising ranks of the approximations remain quasi-optimal; that is, these ranks are bounded by a fixed multiple of corresponding best approximation ranks for a proportional accuracy as in Theorem \ref{thm:stsolve} below.
In our context, we use this method as a solver for Galerkin discretizations that arise in steps of an adaptive scheme, which requires some minor modifications. For $\bv \in \ell_2(\cG)$, we use the notation $\norm{ \bv } = \norm{\bv}_{\ell_2}$.

\subsection{Soft thresholding of hierarchical tensors} 

We first introduce soft thresholding operators for sequences and matrices. Operator for hierarchical tensors is then defined by applying soft thresholding operator successively to each matricization.

For $\mathbf{x} \in \ell_2(\mathcal{I}_1 \times \mathcal{I}_2)$ with countable index sets $\mathcal{I}_1, \mathcal{I}_2$, we denote by  $(\sigma_i(\mathbf{x})) \in \ell_2(\N)$ the non-increasing, non-negative sequence of singular values of $\mathbf{x}$. The singular value decomposition of $\mathbf{x}$ takes the form
\[
\mathbf{x} =\sum_{k=1}^\infty \sigma_k(\mathbf{x}) \mathbf{U}^{(1)}_k \otimes  \mathbf{U}^{(2)}_k
\]
with orthonormal systems $\{ \mathbf{U}^{(i)}_k \}_{k\in\N}$ in $\ell_2(\mathcal{I}_i)$ for $i=1,2$.
Soft thresholding with parameter $\tau \geq 0$ is defined by applying the soft thresholding operation to its singular values, that is,
\begin{equation}
    \label{eq:softmat}
    \ST_{\tau}(\mathbf{x}) = \sum_{k=1}^\infty \max\{ \sigma_k(\mathbf{x}) - \tau, 0 \} \mathbf{U}^{(1)}_k \otimes  \mathbf{U}^{(2)}_k.
\end{equation}

For a hierarchical tensor $\bw$ with suitably numbered effective edges $\{\alpha_t, \alpha_t^c\} \in \mathbb{E}, \, t = 1,\dots,E$, the hierarchical tensor soft thresholding operator $\ST_{\tau}$ as introduced in \cite{BachmayrSchneider:17} is defined as the application of $\ST_{\tau}$ to each matricization, that is,
\begin{equation}
    \label{eq:softtensor}
    \ST_{\tau}(\bw)=\left(\ST_{E, \tau} \circ \cdots \circ \ST_{1, \tau}\right)(\bw),
\end{equation}
where $\ST_{t, \tau}(\bw) = \left( \cM_t^{-1} \circ S_{\tau} \circ \cM_t \right)(\bw)$, and $\cM^{-1}_t$ is the mapping that converts the corresponding matricization back to a tensor. 
Thus soft thresholding is applied to the singular values of matricizations, for which we use the abbreviation
\[
  \sigma_t(\bw) = \sigma\bigl( \cM_t(\bw) \bigr).
\]
For  details on the implementation of the operation \eqref{eq:softtensor}, with its efficiency depending also on a suitable enumeration of the effective edges, we refer to \cite[Sect. 3]{BachmayrSchneider:17}.

Different from hard thresholding, soft thresholding also decreases entries with absolute value above the threshold, which incurs an additional error. However, this operation is non-expansive, as shown in \cite[Proposition 3.2]{BachmayrSchneider:17}: for all $\bv, \bw \in \ell_2(\cG)$ and all $\tau \geq 0$, we have
\[
\norm{ \ST_{\tau}(\bw) - \ST_{\tau}(\bv) } \leq  \norm{ \bw - \bv }.
\]
As a consequence of this property, the composition of $\ST_{\tau}$ with a contractive mapping still yields a convergent fixed-point iteration with a modified fixed point.

In the case of the linear operator equation \eqref{operator-equation}, we consider a contractive mapping $\cT(\bw) = \bw - \omega \left(\bA \bw - \bbf \right)$ with a suitable scaling parameter $\omega > 0$ such that $\norm{ \Id - \omega \bA } \leq \rho < 1$, which implies
\begin{equation}
    \label{eq:contractive}
    \|\cT(\bw) - \cT(\bv)\| \leq \rho \|\bw - \bv\|, \quad \text{for all $\bw, \bv \in \ell_2(\cG)$.}
\end{equation}

As shown in \cite[Lemma 4.1]{BachmayrSchneider:17}, for any $\cT$ as in \eqref{eq:contractive} with unique fixed point $\bu$,
and for any $\tau >0$, there exists a unique $\bu_\tau \in \ell_2(\cG)$ such that $\bu_{\tau} = \ST_{\tau} \left(\cT(\bu_{\tau})\right)$, which satisfies
    \begin{equation}
        \label{eq:bound_soft_contractive}
        (1+\rho)^{-1}\|\ST_{\tau}(\bu) - \bu\| \leq \| \bu_{\tau} - \bu \| \leq (1-\rho)^{-1}\|\ST_{\tau}(\bu) - \bu\|.
    \end{equation}
Moreover, for any $\bw_0 \in \ell_2(\cG)$, for the thresholded fixed-point iteration
\begin{equation}
    \label{eq:soft_iter_basic}
    \bw_{k+1} = \ST_{\tau} \bigl(\cT(\bw_{k})\bigr),
\end{equation}
we have
    \[
        \|\bw_{k} - \bu_{\tau}\| \leq \rho^k \|\bw_0 - \bu_{\tau}\|,
    \]
that is, for any threshold $\tau$, the iteration converges at the same rate as the one for $\cT$ without thresholding, but to a modified fixed point.

To obtain convergence to $\bu$, we thus consider the non-stationary iteration
\begin{equation}
    \label{eq:soft_iter}
    \bw_{k+1} = \ST_{\tau_k} \bigl(\cT(\bw_k)\bigr),
\end{equation}
with a suitable choice of $\tau_k \rightarrow 0$. Evidently, a slow decrease in $\tau_k$ will impede the convergence of the iteration, whereas a rapid reduction in $\tau_k$ may result in iterates with very large tensor ranks. One could choose $\tau_k$ based on the decay of the singular value sequences $\sigma_t(\bu)$ of $\cM_t(\bu)$. In \cite[Theorem 4.4]{BachmayrSchneider:17}, an explicit choice for $\tau_k$ based on knowledge of this decay that yields quasi-optimal rank bounds is provided.
However, the precise decay of the sequences $\sigma_t(\bu)$ is typically not known, and \cite[Algorithm 2]{BachmayrSchneider:17} yields an a posteriori choice of $\tau_k$ that preserves the quasi-optimality of ranks.

\subsection{Application to Galerkin discretizations}\label{sec:galsolve}

As part of our adaptive scheme, we apply a version of the iteration \eqref{eq:soft_iter} with soft thresholding for solving Galerkin discretizations associated to a sequence of finite subsets of $\cG$. For the discussion of this Galerkin solver routine, we now assume such a finite index set $\Lambda \subset \cG$ to be fixed. We define the restriction operator $\Restr_{\Lambda} :\ell_2(\cG) \mapsto \ell_2(\cG)$ as pointwise multiplication by the indicator function of $\Lambda$, that is, $\Restr_\Lambda \bv$ equals $\bv$ on $\Lambda$ and vanishes on $\cG\setminus \Lambda$.

The Galerkin approximation of the solution $\bu$ of $\bA \bu = \bbf$ is defined as the unique solution $\bu_{\Lambda} \in \ell_2(\cG)$ with $\supp \bu_\Lambda \subseteq \Lambda$ of
\begin{equation}
   \bA_\Lambda  \bu_{\Lambda}  = \Restr_\Lambda \bbf.
\end{equation}
with the notation
\begin{equation}
 \bA_\Lambda = \Restr_\Lambda  \bA \Restr_\Lambda, \quad \bbf_\Lambda =  \Restr_\Lambda \bbf.
\end{equation}
As a consequence of the ellipticity and boundedness of operator $\bA$, the Galerkin discretization $\bA_\Lambda$ is again elliptic and bounded with the same constants on the subspace $\operatorname{range} (\Restr_\Lambda)$. Thus $\cT_\Lambda(\bw) = \bw - \omega \left(\bA_\Lambda \bw - \bbf_\Lambda \right)$ on this subspace also satisfies the contractivity property \eqref{eq:contractive} for the same values of $\omega$. Specifically, the spectrum of $\bA_{\Lambda}$ is contained in $[\gamma, \Gamma]$, where by the upper and lower bounds on the diffusion coefficient and \cite[Prop.~3.6]{Bachmayr:23},
\begin{equation}
 \gamma = c_{\Psi}^2 r, \quad  \Gamma = C_{\Psi}^2 (2\|\bar{a}\|_{L_{\infty}} - r),
\end{equation}
with $r$ as in \eqref{UEA} and $c_\Psi, C_\Psi$ as in \eqref{Riesz_basis}. The choice $\omega = 2/(\gamma + \Gamma)$ then yields
\begin{equation}
    \label{eq:omega}
    \norm{ \Restr_\Lambda - \omega \bA_{\Lambda}  }   \leq \frac{\Gamma - \gamma}{\Gamma + \gamma} =: \rho < 1
\end{equation}
independently of $\Lambda$.
The corresponding iterative scheme with $\cT$ replaced by $\cT_\Lambda$,
\begin{equation}
    \label{eq:soft_iter_subspace}
    \bw_{k+1}  = \ST_{\tau_k} \bigl(\cT_\Lambda(\bw_k )\bigr),
\end{equation}
with $\tau_k \to 0$ thus converges to $\bu_\Lambda$ for any $\bw_0$ with $\supp \bw_0 \subseteq \Lambda$.

For finite $\Lambda \subset \cG$ of Cartesian product form $\Lambda = \Lambda_0 \times \cdots \times \Lambda_J$, we have
\begin{equation}\label{eq:cartesianrestr}
  \Restr_\Lambda = \Restr_{\Lambda_0} \otimes \cdots \otimes \Restr_{\Lambda_J},
\end{equation}
so that $\bA_\Lambda$ has the same low-rank structure as $\bA$. It is not difficult to see that also $\sigma_{t,k}( \Restr_\Lambda \bu) \leq \sigma_{t,k}(\bu)$ for $t = 1,\ldots,E$ and $k \in \N$. Note, however, that generally $\Restr_\Lambda \bu \neq \bu_\Lambda$, that is, the restriction of the exact solution $\bu$ need not be equal to the Galerkin solution $\bu_\Lambda$. In particular, the relation between $\sigma_t(\bu_\Lambda)$ and $\sigma_t(\bu)$ is less immediate.

Algorithm \ref{alg_GalerkinSolver} gives the realization of \eqref{eq:soft_iter_subspace} with adaptive adjustment of the thresholding parameter as in \cite[Algorithm 2]{BachmayrSchneider:17}. This scheme does not use a priori knowledge on the low-rank approximability of $\bu$, requiring only a tensor representation of $\bbf$, the action of $\bA$ on such representations, and the bounds $\gamma, \Gamma$ on the spectrum of $\bA$.

\begin{algorithm}
    \caption{\ $(\mathbf w_\Lambda, I) = \STSolve(\bA_\Lambda,\bbf_\Lambda, \varepsilon)$}
    \label{alg_GalerkinSolver}
    \begin{flushleft}
        parameters: $\gamma$, $\Gamma$ as in \eqref{UEA_operator}, $\omega$, $\rho$ as in \eqref{eq:omega}, arbitrary $\nu, \theta \in (0,1)$ and $\tau_0 \geq E^{-1}\omega \|\bbf\|$. \\
        output: $\bw_{\Lambda}$ satisfying $\| \bA_{\Lambda} \bw_{\Lambda} - \bbf_{\Lambda} \| \leq \varepsilon$, final iteration counter $I$        
    \end{flushleft}
    \begin{algorithmic}[1]
        \State $\bw_0 := 0$, $\mathbf{s}_0 := \bA_{\Lambda}\bw_0 - \bbf_{\Lambda}$, $i:=0$
        \While {$\|\mathbf{s}_{i} \| > \varepsilon$}
            \State $\bw_{i+1} = \mathrm{ST}_{\tau_i}(\bw_i - \omega \mathbf{s}_i)$
            \State \label{alg_GalerkinSolver_residual} $\mathbf{s}_{i+1} := \bA_{\Lambda}\bw_{i+1} - \bbf_{\Lambda}$
            \If {$\| \bw_{i+1} - \bw_{i} \| \leq \frac{(1-\rho)\nu}{\Gamma \rho} \norm{\mathbf{s}_{i+1}}$}
                \State $\tau_{i+1} := \theta \tau_i$
            \Else
                \State $\tau_{i+1} := \tau_i$
            \EndIf
            \State $i \leftarrow i+1$
        \EndWhile
        \State $\bw_{\Lambda} := \bw_i$, $I := i$
    \end{algorithmic}
\end{algorithm}

The following result is a direct adaptation of \cite[Theorem 5.1]{BachmayrSchneider:17} to our present case of Galerkin discretizations. It provides quasi-optimality of the approximation ranks of each iterate in solving the given Galerkin system for algebraic or exponential decay of singular values of matricizations.

\begin{theorem}
    \label{thm:stsolve}
    For a given finite support set $\Lambda \subset \cG$, the procedure $\STSolve$ given in Algorithm \ref{alg_GalerkinSolver} produces $\bw_{\Lambda} \in \ell_2(\cG)$ satisfying $\norm{ \bA_{\Lambda} \bw_{\Lambda} - \bbf_{\Lambda} } \leq \varepsilon$ in finitely many steps. In addition, if $\bu$ is the solution to \eqref{operator-equation}, then for the iterates $\bw_i$ for $i =1, \ldots, I$ of Algorithm \ref{alg_GalerkinSolver} we have the following:
    \begin{enumerate}[{\rm(i)}]
        \item If there exist $M>0$ and $0<p<2$ such that $\sigma_{t,j}(\bu_\Lambda) \leq M j^{-\frac1p}$ for $j \in \N$, $t=1,\dots,E$, then for a $\tilde{\rho} \in (0,1)$, with $\varepsilon_i = \tilde{\rho}^i$ and $s = \frac{1}{p}-\frac{1}{2}$, we have
            \[
                \|\bA_{\Lambda}\bw_i - \bbf_{\Lambda}\| \lesssim (J+1)\varepsilon_i, \quad
                \max\limits_{t=1,\dots,E}\rank_t (\bw_i) \lesssim (J+1)^2 M^{\frac{1}{s}} \varepsilon_i^{-\frac{1}{s}}
            \]
            for $i = 1,\ldots, I$.
            The constants depend on $\gamma, \Gamma, \nu, \theta$, $\tau_0$, and $p$.    
        \item If $\sigma_{t,j}(\bu_\Lambda) \leq Ce^{-cj^\beta}$, for $j\in \N$ and $t=1,\dots,E$, with $C,c,\beta>0$, then for a $\tilde{\rho} \in (0,1)$, with $\varepsilon_i = \tilde{\rho}^i$ we have 
            \[
                \|\bA_{\Lambda}\bw_i - \bbf_{\Lambda}\| \lesssim (J+1)\varepsilon_i, \quad
                \max\limits_{t=1,\dots,E}\rank_t (\bw_i ) \lesssim (J+1)^2 \bigl(1+\abs{ \ln\varepsilon_i} \bigr)^{\frac{1}{\beta}}
            \]
            for $i = 1,\ldots, I$.
            The constants depend on $\gamma, \Gamma, \nu, \theta$, $\tau_0$, and on $C$, $c$ and $\beta$.    
    \end{enumerate} 
\end{theorem}

\begin{remark}\label{rem:logJiter}
As noted in \cite[Remark~5.6]{BachmayrSchneider:17}, for the effective convergence rates $\tilde \rho$ in Theorem \ref{thm:stsolve}, we have $\tilde\rho \leq \hat \rho^{1/\log (J+1)} < 1$ with $\hat\rho$ independent of $J$. As a consequence, the number of iterations required for achieving the residual threshold $\varepsilon >0$ in Theorem \ref{thm:stsolve} scales with respect to $\varepsilon$ and $J$ as $\mathcal{O}((\abs{\log \varepsilon} + \log J) \log J)$.
\end{remark}

\begin{remark}
The Galerkin residuals required in line \ref{alg_GalerkinSolver_residual} of Algorithm \ref{alg_GalerkinSolver} are finitely supported and can thus in principle be evaluated exactly. However, to obtain optimal computational costs, these residuals need to be evaluated inexactly so that $ \|\mathbf{s}_{i} - (\bA_\Lambda \bw_{i} - \bbf_\Lambda)\| \leq \delta_i$ with suitable tolerances $\delta_i > 0$. Algorithm \ref{alg_GalerkinSolver} can be modified \cite[Algorithm 3]{BachmayrSchneider:17} to accommodate such inexact residuals and then leads to the same results as in Theorem \ref{thm:stsolve}; we do not go into details at this point to avoid technicalities. Suitable residual approximations that can also be applied to Galerkin discretizations are discussed in Section \ref{sec:resapprox}.
\end{remark}

\section{Adaptive low-rank Galerkin method and optimality}\label{sec:adaptgalerkin}

\label{Sect:adaptive}

In this section, we first introduce the basic idea of the adaptive low-rank Galerkin method to solve the operator equation \eqref{operator-equation} on $\ell_2(\cG)$, then detail its ingredients and analyze convergency and complexity. 

In what follows, for $\bv \in \ell_2(\cG)$ we continue to use the abbreviation $\norm{ \bv } = \norm{\bv}_{\ell_2}$ and introduce the energy norm $\norm{ \bv }_{\bA} = \sqrt{ \langle \bA \bv, \bv \rangle}$, for which we have the norm equivalence
\begin{equation}
    \label{UEA_operator}
    \gamma \|\bu\|^2 \leq \|\bu\|^2_\bA \leq \Gamma \|\bu\|^2, \quad \forall \bu \in \ell_2(\cG),
\end{equation}
where $\gamma = c_{\Psi}^2 r,\, \Gamma = C_{\Psi}^2 (2\|\bar{a}\|_{L_{\infty}} - r)$, with $r$ as in \eqref{UEA} and $c_\Psi, C_\Psi$ as in \eqref{Riesz_basis}.

\subsection{Adaptive Galerkin method}
The idealized adaptive Galerkin scheme is constructed by solving successive Galerkin problems on a sequence of nested finite index sets $\Lambda^k \subset \cG$, $k\geq 0$. Starting from some $\Lambda^0$ that may be empty, these are generated by solving for each $k$ the Galerkin problem on $\Lambda^k$ to obtain $\bu_{k} \in \ell_2(\cG)$ with $\supp\bu_k \subseteq \Lambda^k$ satisfying $ \bA_{\Lambda^{k}} \bu_k = \bbf_{\Lambda^{k}} $. Then $\Lambda^{k+1} \subset \cG$ with $\Lambda^{k} \subset \Lambda^{k+1}$ is chosen to satisfy $\norm{\Restr_{\Lambda^{k+1}} (\bA \bu_k - \bbf)} \geq \mu \norm{ \bA \bu_k - \bbf }$, where $0 < \mu \leq 1$ is fixed and sufficiently small.
Convergence of this idealized scheme is ensured by the following well-known result, see \cite[Lemma 4.1]{CohenDahmenDeVore:01}.

\begin{lemma}
    Let $\mu \in (0,1]$, $\bw \in \ell_2(\cG)$ with $\supp \bw \subseteq \Lambda$, and $\Lambda \subset \tilde{\Lambda} \subset \cG$ such that
    \begin{equation}
        \norm{ \Restr_{\tilde \Lambda} (\bA \bw - \bbf) }  \geq \mu \norm{ \bA \bw - \bbf } .
    \end{equation}
    Let $\bu_{\tilde{\Lambda}}$ be the solution to the Galerkin problem $\bA_{\tilde\Lambda} \bu_{\tilde{\Lambda}} = \bbf_{\tilde \Lambda}$, then
    \begin{equation}
        \| \bu - \bu_{\tilde{\Lambda}} \|_\bA \leq (1-\mu^2\kappa(\bA)^{-1})^{\frac{1}{2}} \| \bu - \bw \|_\bA.
    \end{equation}
\end{lemma}

From a practical point of view, required Galerkin approximations $\bu_k$ in the first step can only be computed up to certain error tolerances that need to be accounted for. Furthermore, a main difficulty is that residuals $\bA \bu_k - \bbf$ generally have infinite support. Therefore, in each step this residual needs to be replaced by a finitely supported approximation $\br_k$ of $\bA \bu_k - \bbf$.

Taking these additional aspects into account, in \cite{GantumurHarbrechtStevenson:07} this basic adaptive scheme is analysed in the context of direct sparse approximations of $\bu$, where arbitrary subsets of $\cG$ can be selected. 
In this case, the performance of the method is compared to best $N$-term approximations of $\bu$, which is obtained by selecting $N$ indices corresponding to entries of largest absolute value of $\bu$. Their convergence can be quantified in terms of the approximation classes
\begin{equation}
    \label{eq:approximation_space}
   \cA^s(\cG)  = \bigl\{  \bv \in \ell_2(\cG) \colon \norm{\bv}_{\cA^s(\cG)} := \sup_{N \in \mathbb{N}_0}(N+1)^s \inf_{\#\supp \bv_N \leq N}\| \bv- \bv_N\| < \infty \bigr\}
\end{equation}
for $s>0$, so that for all $\bv \in \cA^s(\cG)$ and $\varepsilon > 0$, there exists $\tilde\bv$ with $\| \bv- \tilde\bv\| \leq \varepsilon$ and
\[
   \# \supp \tilde\bv  \leq  \varepsilon^{-\frac1s} \norm{\bv}_{\cA^s(\cG)}^{\frac1s} \,.
\]

With appropriately adjusted error tolerances, and provided that $\mu$ is chosen sufficiently small and $\Lambda^{k+1}\supseteq \Lambda^k$ is chosen as the smallest index set such that $\norm{ \Restr_{\Lambda^{k+1}} (\bA \bu_k - \bbf)} \geq \mu \norm{ \bA \bu_k - \bbf }$, the method in \cite{GantumurHarbrechtStevenson:07} is shown to produce quasi-optimal selections of $\Lambda^k$ in relation to the achieved error proportional to $\norm{ \bA \bu_k - \bbf }$, that is, for each given tolerance $\varepsilon>0$, the method produces a $\bu_k$ supported on $\Lambda^k$ such that 
\[
 \norm{\bu_k - \bu} \eqsim \norm{\bA \bu_k - \bbf} \lesssim \varepsilon , \qquad \#\Lambda^k \lesssim \varepsilon^{-\frac1{s} } \norm{\bu}_{\cA^s}^{\frac1{s}} \,.
\]

However, for low-rank approximation, we need to construct $\Lambda^{k}$ of Cartesian product form $\Lambda^k = \Lambda^k_0 \times \cdots\times\Lambda^k_J$ as in \eqref{eq:cartesianrestr}. When all operations are performed in low-rank format, both the storage required for approximations and the complexity of the entire scheme depend not on the product, but on the \emph{sum} of mode sizes $\sum_{j=0}^{J}\#\Lambda^{k}_j$, and we accordingly aim to obtain quasi-optimal bounds on these quantities.
While this is achieved by the adaptive low-rank method developed in \cite{BachmayrDahmen:14}, there this requires an intermittent coarsening of discretizations that removes extraneous degrees of freedom. 
As a main contribution in this paper, we exploit the particular structure of our novel approximation format for parametric problems to achieve quasi-optimality of discretizations without coarsening.

In addition to the restriction to product index sets, identifying them by direct comparison of absolute values of entries would be too inefficient, and following \cite{BachmayrDahmen:14}, we thus instead identify product index sets from lower-dimensional quantities termed \emph{contractions}.
We next recall the relevant definitions from \cite[Sec.~3.2.1]{BachmayrDahmen:14} and introduce the required notation. 

For $j = 1,\ldots,J$, let $\nu = (\nu_1,\ldots,\nu_J) \in \cG_1 \times \cdots \times \cG_J$ and $\Lambda_1 \times \cdots \times \Lambda_J \subset \cG_1 \times \cdots \times \cG_J$, we define
\begin{gather}
    \label{eq:check_nu}
    \check{\nu}_j = (\nu_1,\dots,\nu_{j-1},\nu_{j+1},\dots,\nu_J), \\
    \label{eq:check_Lambda}
    \check{\Lambda}_j = \Lambda_1 \times \cdots \times \Lambda_{j-1} \times \Lambda_{j+1} \times \cdots \times \Lambda_J
\end{gather}
to refer to the corresponding multi-index and product set with component $j$ deleted, respectively. We emphasize that index \eqref{eq:check_nu} and set \eqref{eq:check_Lambda} correspond only to the parametric tensor modes. Since the index $(\bar{\nu},\lambda)$ and set $\Lambda_0$ corresponding to the sparse expansion mode differ from those of the parametric modes, we write them separately. For $\bv \in \ell_2(\cG)$, contractions are defined as 
\begin{equation}
    \label{eq:contraction}
    \begin{aligned}
        & \pi^{(0)}(\bv) = \left(\pi^{(0)}_{(\bar\nu,\lambda)}(\bv)\right)_{(\bar\nu,\lambda) \in \cG_0} = \Biggl( \biggl(\sum_{\nu  \in \cG_1 \times \cdots\times \cG_J} \abs{\bu_{(\bar\nu,\lambda),\nu}}^2 \biggr)^{\frac{1}{2}} \Biggr)_{(\bar\nu,\lambda) \in \cG_0}, \\
        & \pi^{(j)}(\bv) = \left(\pi^{(j)}_{\nu_j}(\bv)\right)_{\nu_j \in \cG_j} = \Biggl( \biggl(\sum_{(\bar\nu,\lambda) \in \cG_0}\sum_{\check{\nu}_j \in \check{\cG}_j}\abs{\bu_{(\bar\nu,\lambda),\nu}}^2 \biggr)^{\frac{1}{2}} \Biggr)_{\nu_j \in \cG_j}, \quad j \in \{1,\dots,J\},
    \end{aligned}
\end{equation}
where we write
\[  \check{\cG}_j = \cG_1 \times \cdots \times \cG_{j-1} \times \cG_{j+1} \times \cdots \times \cG_J \,.    \]
These quantities can be computed efficiently from singular value decompositions of matricizations of $\bv$ at costs proportional to $\#\Lambda_j \rank\cM_{\{j\}}(\bv)$, see \cite[Prop.~2]{BachmayrDahmen:14}.

For $\bv \in \ell_2(\cG)$ such that $\pi^{(j)}(\bv) \in \cA^{s_j}(\cG_j)$ with $s_j >0$ for $j= 0,\ldots,J$, for $\varepsilon>0$ we thus aim to identify product index sets $\Lambda_\varepsilon = \Lambda_{\varepsilon,0} \times \cdots \times \Lambda_{\varepsilon,J}$ such that $\norm{\bv - \bv_{\Lambda_\varepsilon}} \leq \varepsilon$ that satisfy the quasi-optimal bound
\[
    \sum_{j=0}^{J} \#\Lambda_{\varepsilon,j} \lesssim \sum_{j=0}^{J}  \bignorm{\pi^{(j)}(\bv)}_{\cA^{s_j}}^{\frac{1}{s_j}} \varepsilon^{-\frac{1}{s_j}} \,.
\]

An adaptive low-rank Galerkin method that does not use coarsening of discretizations is given in Algorithm \ref{alg_adaptive}. It comprises three fundamental components specifically designed for the hierarchical tensor format and adapted to the present class of parametric problems:
\begin{enumerate}[{\rm(I)}]
    \item \emph{Galerkin solver:} Provide a Galerkin approximation $\bu_k \in \ell_2(\cG)$ supported on $\Lambda^k$ such that for the given error tolerance $\eta$, 
        \begin{equation}
            \label{eq:Galerkin_error_estimate}
            \| \bA_{\Lambda^k} \bu_{k} - \bbf_{\Lambda^k} \| \leq \eta.
        \end{equation}
        For this step, we rely on the procedure described in Section \ref{sec:galsolve}.
    \item \emph{Residual approximation:} Return a finitely supported approximation $\br_{k}$ of $\bA  \bu_{k} - \bbf$ with a given tolerance $\xi$ and $\tilde{\Lambda}^{k}$, such that $\supp \br_k \subset \tilde{\Lambda}^{k}$ and
        \begin{equation}
            \label{eq:resi_appro_error_estimate}
            \| \br_{k} - (\bA \bu_{k} - \bbf) \| \leq \xi .
        \end{equation}
        This step is considered in more detail in Section \ref{sec:resapprox}.
    \item \emph{Discretization refinement in tensor modes:} For a given parameter $\alpha \in (0,1)$, return a product index set $\Lambda^{k+1}$ with $\Lambda^{k} \subset \Lambda^{k+1} \subset \operatorname{supp}(\br_{k})$, such that
        \begin{gather}
        	\label{update_rule}
            \norm{  \Restr_{\Lambda^{k+1}} \br_{k} } \geq \alpha \norm{ \br_{k}  } , \\
            \label{update_rule_optimality}
            \sum_{j=0}^{J}\#(\Lambda^{k+1}_j \backslash \Lambda^{k}_j) \lesssim \sum_{j=0}^{J} \bignorm{\bA \bu_k - \bbf}^{-\frac{1}{s_j}} \bignorm{\pi^{(j)}(\bu)}_{\cA^{s_j}}^{\frac{1}{s_j}} .
        \end{gather}
        This step requires particular care in our case to ensure that no later coarsening of discretizations is required. We consider this question in more detail in Section \ref{sec:refinement}.
\end{enumerate}

To achieve \eqref{eq:Galerkin_error_estimate}, we use the Galerkin solver based on the Richardson iterative method with a soft thresholding operator for rank reduction in hierarchical tensors described in Algorithm \ref{alg_GalerkinSolver}. 
We provide a detailed analysis of the further components in the following subsections.

\begin{algorithm}
    \caption{\; $\bu_\varepsilon = \mathrm{ALRSolve}(\varepsilon)$}
    \label{alg_adaptive}
    \begin{flushleft}
        parameters: $\delta \in (0,\alpha)$, $\mu= \frac{\alpha+\delta}{1-\delta} < \kappa(\bA)^{-\frac{1}{2}}$ and 
        $0<\eta < \frac{(1-\delta)(\alpha-\delta)}{1+\delta}\kappa(\bA)^{-1}$.\\
        output: $\bu_{\varepsilon}$ satisfying $\|\bA \bu_{\varepsilon} - \bbf\| \leq \varepsilon$
    \end{flushleft}
    \begin{algorithmic}[1]
    \State $k:=0$, $\Lambda^0 := \emptyset$, $\bu_0 := 0$, $\varepsilon_{-1} := \|\bbf\|$
    \State $\xi:= \theta \varepsilon_{-1}$
    \While {$\xi > \delta \|\br_0\|$}
        \State $\xi := \xi/2$
        \State $\br_0, \tilde{\Lambda}^0 := \ResApprox(\bu_0,\bA,\bbf,\Lambda^0, \xi)$ \Comment See Section \ref{sec:resapprox}
    \EndWhile
    \While {$\varepsilon_k := \| \br_{k} \| + \xi > \varepsilon$}
        \State $\Lambda^{k+1} := \Expand(\br_k,\Lambda^k,\tilde{\Lambda}^{k},\alpha)$ \label{alg:Adap_expand} \Comment See Section \ref{sec:refinement}
        \State $\bu_{k+1} := \STSolve(\bA_{\Lambda^{k+1}}, \bbf_{\Lambda^{k+1}}, \eta \|\br_k\|)$ \label{alg:Adap_solve_relative_error}  \Comment See Section \ref{sec:galsolve}
        \State $\xi := \theta \varepsilon_k$
        \While {$\xi > \delta \|\br_{k+1}\|$} \label{alg:Adap_resi_relative_error}
            \State $\xi := \xi/2$
            \State $\br_{k+1}, \tilde{\Lambda}^{k+1} := \ResApprox(\bu_{k+1},\bA,\bbf,\Lambda^{k+1}, \xi)$  \Comment See Section \ref{sec:resapprox}
        \EndWhile
        \State $k \leftarrow k+1$
    \EndWhile
    \State $\bu_{\varepsilon} := \bu_k$
    \end{algorithmic}
\end{algorithm}

\subsection{Residual approximation}\label{sec:resapprox}

Let $\bw_{\Lambda}$ with $\supp \bw_\Lambda \subseteq \Lambda$ be an approximate solution to $\bA_\Lambda \bu_\Lambda =\bbf_\Lambda$ with Galerkin error $\norm{\bA_\Lambda \bw_\Lambda - \bbf_\Lambda} \leq \eta$. In order to find a finitely supported approximation to $\bA\bw_{\Lambda}-\bbf$, we determine an extended product index set $\tilde{\Lambda} \supset \Lambda$ based on a fixed residual tolerance $\xi$ and construct an approximation $\br_\Lambda$ of $\bA_{\tilde{\Lambda}}\bw_\Lambda - \bbf_{\tilde{\Lambda}}$ such that $\supp \br_\Lambda \subset \tilde{\Lambda}$ and $\norm{ \br_{\Lambda} - (\bA \bw_{\Lambda} - \bbf) } \leq \xi$. This procedure is referred to as
$\ResApprox(\bw_{\Lambda}, \bA, \bbf, \Lambda, \xi)$. 

By iteratively halving the value of $\xi$ in $\ResApprox(\bw_{\Lambda}, \bA, \bbf, \Lambda, \xi)$ until $\xi \leq \delta \norm{\br_{\Lambda}}$ for a given $\delta$, we obtain the desired approximation $\br_{\Lambda}$ which satisfies \eqref{eq:resi_appro_error_estimate} and thus also the relative error bound
\[  \norm{\br_\Lambda - (\bA\bw_{\Lambda}-\bbf)} \leq \delta \norm{\br_\Lambda}  \]
as required in Algorithm \ref{alg_adaptive}.

We recall the hierarchical tensor decomposition of $\bA$ as in \eqref{operator-TT-form} and \eqref{def_A_tensor} in terms of the components
\[
    \bA = \bar{\bA} \otimes \bI_1 \otimes \cdots \otimes \bI_J + \bB_1 \otimes \bN_1 \otimes \bI_2 \otimes \cdots \otimes \bI_J + \cdots + \bB_J \otimes \bI_{1} \otimes \cdots \bI_{J-1} \otimes \bN_J,
\]
with the components associated to the spatial-parametric zeroth mode given by
\[  
\bar{\bA} = \bM_0 \otimes \bA_0 + \sum_{i = J+1}^\infty (\bM_i \otimes \bA_i), \qquad \bB_{j} = \bM_0 \otimes \bA_j, \quad j=1,\dots,J.
\]
The set product set $\tilde{\Lambda}$ is built by identifying its factors $\tilde{\Lambda}_j \subset \cG_j$, $j=0,\ldots,J$, using the contractions $\pi^{(j)}(\bw_\Lambda)$, $j = 0, \ldots,J$, for controlling error tolerances, as described in \cite[Sec.~4.1]{BachmayrDahmen:14}.

We first consider the index sets $\tilde{\Lambda}_j$ for the parametric modes $j=1,\ldots,J$. Due to the three-term recursion for the orthonormal Legendre polynomials \eqref{eq:threeterm},
the matrix $\bN_j$ is bidiagonal for $j = 1,\dots,J$. Specifically, the entry of $\br_{\Lambda}$ with index $\big((\bar{\nu}',\lambda'),\nu'_1,\ldots,\nu'_J\big)$ can be non-zero only when there exists a $j \in \{1,\ldots,J\}$ and a $\nu_j \in \Lambda_j$ such that $\nu^{\prime}_j = \nu_j \pm 1$. Consequently, this allows for the simple construction of $\tilde{\Lambda}_j$ as 
\begin{equation}
    \label{tilde_Lambda_j}
    \tilde{\Lambda}_j = \{\nu_j \pm 1 : \nu_j \in \Lambda_j\}\cap \N_0, \quad j = 1,\dots,J.
\end{equation}

The main difficulty lies in the construction of $\tilde{\Lambda}_0$ corresponding to the sparse expansion mode, which involves enriching both spatial discretizations and the product Legendre polynomial expansions. 
In view of \cite[Sec.~4.1]{BachmayrDahmen:14}, this requires approximating the action of $\bar\bA$ and $\bB_j$, $j=1,\ldots, J$, on the contraction $\pi^{(0)}(\bw_\Lambda)$.
There exist different methods for performing this approximation and identifying $\tilde\Lambda_0$ that are applicable in particular under different assumptions on the functions $\theta_j$, $j \in \cI$, in the expansion \eqref{eq:affine}. The choice of the method for producing $\#\tilde\Lambda_0$ is not essential for the developments in the present work, but has an impact on the total computational costs of $\ResApprox$. We will investigate this point in more detail in future work, and here we thus only outline several available constructions.

A standard framework for sparse approximation of residuals is based on \emph{$s^*$-compressibility} of operators \cite{CohenDahmenDeVore:01}. The compressibility of $\bB_j$ then follows directly from that of the spatial component $\bA_j$ for $j=1,\ldots,J$, which is provided by standard results; see, for example, \cite{CohenDahmenDeVore:01,Stevenson:02}. For $\{ \theta_j \}$ with global supports on $D$ (as, for example, in Karhunen-Lo\`eve expansions of random fields), compressibility of $\bar\bA$ can be deduced from the results in \cite{Gittelson:13,Gittelson:14}. In general, however, in this setting one does not obtain optimal computational costs, except when the order is limited by the spatial discretization.

Better compressibility of $\bar\bA$ and accordingly improved computational costs can be obtained for $\theta_j$ with multilevel structure as in Assumption \ref{multilevel_structure}. The following observation from \cite{BachmayrCohenDahmen:18} is the basis for the improved approximations of $\bar\bA$.

\begin{prop}
    \label{Prop:semi_residual_appro}
    Let Assumption \ref{multilevel_structure} hold. Then for $\ell \in \N$,
    \begin{equation}
        \label{eq:semi_residual_appro}    
         \biggnorm{ \bar{\bA} - \sum_{\substack{(\ell,k) \in \bigcup_{n < L} \{ n\} \times \mathcal I_n \\ j(\ell,k) > J}} \mathbf{M}_{j(\ell,k)} \otimes \mathbf{A}_{ j(\ell,k) } } \lesssim  2^{- \alpha L }
     \end{equation}
    with constant depending on those in Assumption \ref{multilevel_structure} and on $c_{\Psi}$, $C_{\Psi}$ in \eqref{Riesz_basis}.
\end{prop}

This leads to approximations of $\bar\bA$ by levelwise truncation of the sum in \eqref{eq:affine}, which in  \cite{BachmayrCohenDahmen:18} is combined with standard compression for the operators $\bA_0$ and $\bA_j$, $j>J$, to show $s^*$-compressibility of $\bar\bA$. However, this does not lead to optimal computational costs and requires strong smoothness requirements on the functions $\theta_j$ and the spatial wavelet basis.

An alternative approach based on Assumption \ref{multilevel_structure} is developed in \cite{BachmayrVoulis:22}, based on the observation that a semi-discrete approximation in the parametric variables based on Proposition \ref{Prop:semi_residual_appro} can be combined with a fixed level of local spatial refinement in the spatial wavelet discretizations to achieve any desired relative error. This scheme leads to a method of optimal computational complexity. In \cite{BachmayrVoulis:22}, the additional assumption that the wavelet basis functions are in $H^2(D)$ is used; more recent results in \cite{BEEV:24}, however, show that this requirement can be removed.

For our numerical tests, we use a simplified version of the method described in \cite[Sect. 4]{BachmayrVoulis:22} to generate $\tilde\Lambda_0$ and then directly approximate $\mathbf{A}_{\tilde\Lambda} \mathbf{w}_\Lambda - \bbf_{\tilde \Lambda}$ based on Proposition \ref{Prop:semi_residual_appro} by directly applying the spatial discretization matrices. While this leads to the same $\tilde\Lambda$ as the method from \cite{BachmayrVoulis:22}, for optimal computational costs the latter additionally requires tree-based refinement of wavelet discretizations and that all operations are carried out hierarchically using this tree structure. Since the incorporation of these further techniques into the present adaptive low-rank method are not essential for what follows, we address this point in a forthcoming work.

\subsection{Discretization refinement}\label{sec:refinement}

Let $\bu \in \ell_2(\cG)$ be the solution to \eqref{operator-equation} and let $\pi^{(0)}(\bu) \in \cA^{s_1}$, $\pi^{(j)}(\bu) \in \cA^{s_2}$ for $j=1,\dots,J$, which quantify the approximabilities of sparse expansion mode frame and parametric mode frames, respectively. Let $\bw_{\Lambda}$ with $\supp \bw_\Lambda \subseteq \Lambda$ be an approximate Galerkin solution $\norm{\bA_\Lambda \bw_\Lambda - \bbf_\Lambda} \leq \eta$. Moreover, let $\br_\Lambda \in \ell_2(\cG)$ with finite support $\tilde{\Lambda}$ satisfy $\norm{\br_\Lambda - (\bA\bw_{\Lambda}-\bbf)} \leq \delta \norm{\br_\Lambda}$, where $0<\delta<\alpha$ and $\alpha > 0$ as in \eqref{update_rule_sect4}. We aim to construct a quasi-optimal product set $\bar{\Lambda}$ with $ \Lambda \subset \bar{\Lambda} \subset \tilde{\Lambda}$, such that
\begin{equation}
    \label{update_rule_sect4}
    \norm{\Restr_{\bar{\Lambda}}\br_{\Lambda} } \geq \alpha\norm{\br_{\Lambda}},
\end{equation}
and simultaneously bound the total number of added indices in all tensor modes, that is, we aim to show an estimate of the form
\begin{equation}
    \label{update_rule_optimality_sect4}    
    \sum_{j=0}^{J}\#(\bar\Lambda_j \backslash \Lambda_j) \lesssim \bignorm{\bA \bw_\Lambda - \bbf}^{-\frac{1}{s_1}} \, \bignorm{\pi^{(0)}(\bu)}_{\cA^{s_1}}^{\frac{1}{s_1}} + \sum_{j=1}^{J}  \bignorm{\bA \bw_\Lambda - \bbf}^{-\frac{1}{s_2}} \, \bignorm{\pi^{(j)}(\bu)}_{\cA^{s_2}}^{\frac{1}{s_2}}.
\end{equation}

In the framework of sparse approximations without tensor product structure as in \cite{CohenDahmenDeVore:01, GantumurHarbrechtStevenson:07}, the natural approach to refine $\Lambda$ is adding the minimal number of indices corresponding to the largest entries of $\br_\Lambda$ to $\Lambda$, such that \eqref{update_rule_sect4} is satisfied. This process of discarding basis indices from approximate solutions is called \emph{coarsening}. 
In \cite{CohenDahmenDeVore:01}, recurrent coarsening steps are added after a certain number of refinements. In \cite{GantumurHarbrechtStevenson:07}, the authors give a modified adaptive Galerkin method that is quasi-optimal without requiring coarsening by choosing sufficiently small $\alpha < \kappa(\bA)^{-\frac{1}{2}}$.

In cases where $\Lambda = \Lambda_0 \times \cdots \times \Lambda_J$ has a Cartesian product structure, we are thus interested in iteratively refining $\Lambda$ while guaranteeing quasi-optimality without recurrent coarsening.
In \cite{BachmayrDahmen:14}, another type of adaptive algorithm based on inexact Richardson iteration is considered that after several steps of refinement performs a coarsening based on contractions. 
While the construction ensures quasi-optimal sizes of the discretization in the tensor modes, coarsening with a sufficiently large error tolerance is required, which in turn can be done only when a certain error reduction has been achieved. Especially for large numbers of tensor modes, this interplay leads to quantitatively unfavorable performance. Moreover, an adaptive Galerkin method for hierarchical tensor format was introduced in \cite{AliUrban:20}, where recurrent coarsening also played a crucial role for optimality of discretizations. 

In what follows, we first introduce the expansion method given in Algorithm \ref{alg_expand} in details that refines $\Lambda$ to $\bar{\Lambda}$ in a tensor product structure, which is based on contractions of approximate residual $\br_\Lambda$. In Theorem \ref{thm:bound_bar_Lambda} we then prove that the total number of added indices in all tensor modes, $\sum_{j=0}^{J} \#(\bar{\Lambda}_j \backslash \Lambda_j)$, can be related to that of the optimal product set $\widehat\Lambda$ in the sense that it satisfies \eqref{update_rule_sect4} while it has a near-minimal number of added indices in all tensor modes. In Theorem \ref{thm:optimal_set_size}, with setting $\alpha < \kappa(\bA)^{-\frac{1}{2}}$, we show the quasi-optimality of set $\Lambda^*$, where $\Lambda^* = \Lambda_0^* \times \cdots \times \Lambda_J^* \subset \cG$ is a minimizer of $\sum_{j=0}^J \#(\Lambda_j^* \backslash \Lambda_j)$ among the product sets such that
\[
    \norm{\Restr_{\Lambda^*}(\bA \bw_\Lambda -\bbf)} \geq \alpha \norm{\bA \bw_\Lambda - \bbf} .
\]
Altogether, these results establish quasi-optimality of $\bar{\Lambda}$ as constructed by Algorithm \ref{alg_expand}, which is illustrated in Theorem \ref{thm:complexity}.

We recall the following proposition from \cite[Prop.~2]{BachmayrDahmen:14}, which illustrates some properties of contractions that are crucial for our expansion method. Property (ii) bounds coarsening errors in terms of contractions, while property (iii) provides a way of computing contractions from singular value decompositions of matricizations without high-dimensional summations.
\begin{prop}
    \label{prop:contractions}
    Let $\bv \in \ell_2(\cG)$.
    \begin{enumerate}[{\rm(i)}]
        \item We have $\|\bv\| = \| \pi^{(j)}(\bv) \|$, $j = 0,1,\dots,J$.
        \item Let $\Lambda \subset \cG$; then \[ \|\bv - \Restr_{\Lambda}\bv \|^2 \leq \sum_{(\bar\nu_,\lambda) \in \cG_0 \backslash \Lambda_0}|\pi^{(0)}_{(\bar\nu,\lambda)}(\bv)|^2 + \sum_{j=1}^{J} \sum_{\nu_j \in \cG_j \backslash \Lambda_j} |\pi^{(j)}_{\nu_j}(\bv)|^2, \]
        \item For $j=0,1,\ldots,J$, let $\mathbf{U}^{(j)}$ be the mode frames of a higher-order singular value decomposition, and let $\sigma^{(j)}$ be the corresponding sequences of singular values. Then 
        \begin{gather}
            \pi^{(0)}_{(\bar\nu,\lambda)}(\bv) = \biggl( \sum_{k} |\mathbf{U}^{0}_{(\bar\nu,\lambda),k}|^2 |\sigma^{(0)}_k|^2 \biggr)^{\frac{1}{2}}, \quad  (\bar\nu,\lambda) \in \cG_0, \\
            \pi^{(j)}_{\nu_j}(\bv) = \biggl( \sum_{k} |\mathbf{U}^{(j)}_{\nu_j,k}|^2 |\sigma^{(j)}_k|^2 \biggr)^{\frac{1}{2}}, \quad \nu_j \in \cG_j, \quad j \in \{1,\dots,J\}.
        \end{gather}
    \end{enumerate}
\end{prop}

To see how to construct a product index set that controls coarsening errors through the tails of contraction sequences, and since we need to construct an index set $\bar{\Lambda} \supset \Lambda$, we focus on the set of contractions associated with indices $\nu_j \in \tilde{\Lambda}_j \backslash \Lambda_j$ (referred to as the \textit{new indices}) for all tensor modes:
\begin{equation}
    \label{contractions_new}    
   \bigl \{ \pi^{(0)}_{(\bar\nu,\lambda)}(\br_\Lambda): (\bar\nu,\lambda) \in \tilde{\Lambda}_0 
    \backslash \Lambda_0 \bigr\} \cup \bigl\{  \pi^{(j)}_{\nu_j}(\br_\Lambda) : \nu_j \in \tilde{\Lambda}_j \backslash \Lambda_j, j = 1,\dots,J \bigr\}.
\end{equation}
We sort these contractions in a nonincreasing order, and denote the rearranged sequence by $(\pi^{*}_n)_{n\in\N}$. By plugging $\br_\Lambda$ into property $(ii)$ of Proposition \ref{prop:contractions}, we conclude
\[
    \norm{\br_\Lambda}^2 \leq \sum_{n=1}^\infty \abs{ \pi_{n}^{*}}^2 + \norm{\Restr_{\Lambda}\br_\Lambda}^2,
\]
where $\norm{\Restr_{\Lambda}\br_\Lambda}$ is computationally feasible, and given that $\norm{\br_\Lambda - (\bA\bw_{\Lambda}-\bbf)} \leq \delta \norm{\br_\Lambda}$, it is bounded proportionally to the Galerkin error, as
\[
    \frac{1}{1+\delta}\bignorm{\bA_\Lambda \bw_\Lambda - \bbf_\Lambda} \leq \norm{\Restr_{\Lambda}\br_\Lambda} \leq \frac{1}{1-\delta}\bignorm{\bA_\Lambda \bw_\Lambda - \bbf_\Lambda}.
\]
Accordingly, for a fixed $\alpha \in (0,1)$, there exists a minimal $M \in \N$ such that
\begin{equation}
    \label{update_rule_contr}
    \sum_{n=1}^M \abs{ \pi_{n}^{*}}^2 + \norm{\Restr_{\Lambda}\br_\Lambda}^2 
    \geq \alpha^2 \norm{\br_\Lambda}^2.
\end{equation}

Similarly, we now redistribute the indices in $\tilde\Lambda_0, \tilde \Lambda_1, \ldots, \tilde\Lambda_J$ corresponding to $\pi^{*}_1,\ldots, \pi^{*}_M$ to the respective component index sets, along with component index sets of $\Lambda$, to form $\Lambda^{\mathrm{min}}_0(\br_{\Lambda};M), \ldots, \Lambda^{\mathrm{min}}_J(\br_{\Lambda};M)$. For their Cartesian product, we write
\begin{equation}
	\label{set}
    \Lambda^{\mathrm{min}}(\br_{\Lambda}; M) = \Lambda_0^{\mathrm{min}}(\br_{\Lambda};M) \times \cdots \times \Lambda_J^{\mathrm{min}}(\br_{\Lambda};M),
\end{equation}
where
\[
    \Lambda_j^{\mathrm{min}}(\br_{\Lambda};M) = \Lambda_j \cup \{ \nu^{*}_i: \nu^{*}_i \in \tilde{\Lambda}_j, i=1,\ldots,M \}, \quad j = 0,1,\ldots,J.
\]
By construction, 
\begin{equation}
    \label{bound_of_Lambdat_L}
    \sum_{j = 0}^J \#(\Lambda_j^{\mathrm{min}}(\br_{\Lambda};M) \backslash \Lambda_j) = M,
\end{equation}
and
\begin{equation}
    \label{optimality_of_Lambda_min}
    \sum_{j =0}^J \sum_{\nu_j \in \tilde{\Lambda}_j \backslash  \Lambda_j^{\mathrm{min}}(\br_{\Lambda};M)} |\pi^{(j)}_{\nu_j}(\br_{\Lambda})|^2 = \min_{\hat{\Lambda}} \biggl\{\sum_{j=0}^J \sum_{\nu_j \in \tilde{\Lambda}_j \backslash  \hat{\Lambda}_j} |\pi^{(j)}_{\nu_j}(\br_{\Lambda})|^2 \biggr\},
\end{equation}
where $\hat{\Lambda}$ ranges over all product sets $\bigtimes_{j=0}^J \hat{\Lambda}_j \supset \Lambda$ satisfying $\sum_{j=0}^J \#(\hat{\Lambda}_j \backslash \Lambda_j) \leq M$.

Recalling \eqref{eq:contraction}, we note that the contraction for one tensor mode involves summations over index sets of all other tensor modes. Therefore, there are some summands associated with indices in $\tilde{\Lambda} \backslash \Lambda^{\mathrm{min}}(\br_{\Lambda}; M)$, leading to
\[
    \sum_{n=1}^M \abs{ \pi_{n}^{*}}^2 + \norm{\Restr_{\Lambda}\br_\Lambda}^2 \geq \norm{\Restr_{\Lambda^{\mathrm{min}}(\br_{\Lambda};M)} (\br_{\Lambda})}^2.
\]
As a result, we cannot deduce that $\norm{\Restr_{\Lambda^{\mathrm{min}}(\br_{\Lambda};M)} (\br_{\Lambda})} \geq \alpha \norm{\br}$ from condition \eqref{update_rule_contr}. Here, we can verify it numerically, as in line \ref{alg_expand_ifcond} of Algorithm \ref{alg_expand}, and perform further modifications to construct the index set if it is not already satisfied.

Let us recall that our goal is to construct a quasi-optimal set $\bar{\Lambda} = \bigtimes_{j=0}^J\bar{\Lambda}_j$ with $\Lambda \subset \bar{\Lambda} \subset \tilde{\Lambda}$. We note that from the definition of $\tilde{\Lambda}_j, \, j=1,\ldots,J$, given in \eqref{tilde_Lambda_j} and $\tilde{\Lambda} \supset \Lambda$, it holds that
\[
    \#(\tilde{\Lambda}_j \backslash \Lambda_j) = 1, \, j \in \{1,\dots,J\}.
\]
We emphasize that this property only holds true in our specific setting of parametric elliptic problems, where we use Legendre polynomial discretizations of the parametric space and start from empty index sets.

Let $\bar{\Lambda}_j = \tilde{\Lambda}_j, \, j=1,\ldots,J$, which are the largest possible extension index sets corresponding to parametric modes. We consider a nonincreasing rearrangement $(\pi^{**}_n)_{n\in\N}$
of the set of contractions corresponding to new indices only for the sparse extension mode
\begin{equation}
    \label{contractions_sparse}    
   \bigl \{ \pi^{(0)}_{(\bar\nu,\lambda)}(\br_\Lambda): (\bar\nu,\lambda) \in \tilde{\Lambda}_0 
    \backslash \Lambda_0  \bigr \}.
\end{equation}
Noting that $\norm{\pi^{(0)}(\br_\Lambda)} = \norm{\br_\Lambda}$, for a given $\alpha \in (0,1)$, we take the minimal $K \in \N$ such that
\begin{equation}
    \label{update_rule_sparse}
    \sum_{n=1}^K \abs{ \pi_{n}^{**}}^2 + \sum_{(\bar\nu,\lambda) \in \Lambda_0} |\pi^{(0)}_{(\bar\nu,\lambda)}(\br_\Lambda)|^2 \geq \alpha^2 \norm{\br_\Lambda}^2.
\end{equation}
We construct $\bar{\Lambda}_0$ by distributing these $K$ indices corresponding to $\pi^{**}_1,\ldots, \pi^{**}_K$ to $\bar\Lambda_0$, along with $\Lambda_0$. Rather than the case of $\Lambda^{\mathrm{min}}(\br_{\Lambda};M)$, since $\bar{\Lambda}_j = \tilde{\Lambda}_j$ for $j \in \{1,\dots,J\}$, we have
\[
    \sum_{n=1}^K \abs{ \pi_{n}^{**}}^2 + \sum_{(\bar\nu,\lambda) \in \Lambda_0} |\pi^{(0)}_{(\bar\nu,\lambda)}(\br_\Lambda)|^2 = \norm{\Restr_{\bar{\Lambda}}\br_{\Lambda}}^2
\]
thus requirement \eqref{update_rule_sect4} is satisfied. Additionally, we have
\[
    \#(\bar{\Lambda}_0 \backslash \Lambda_0) + \sum_{j=1}^{J}\#(\bar{\Lambda}_j \backslash \Lambda_j) = K +J.
\]

We summarize the above coordinates refinement procedure in Algorithm \ref{alg_expand}. In order to show quasi-optimality of the resulting set $\bar{\Lambda}$, we first bound the total number of added degrees of freedom in all tensor modes, $\sum_{j=0}^J \#(\bar{\Lambda}_j \backslash \Lambda_j)$, in terms of $\widehat\Lambda$, which we define as the index set that is in this sense smallest while still satisfying \eqref{update_rule_sect4}.

\begin{algorithm}
	\caption{\; $\bar\Lambda = \Expand(\br_\Lambda, \Lambda, \tilde{\Lambda}, \alpha)$}
    \label{alg_expand}
    \begin{flushleft}
        output: $\bar{\Lambda}$ satisfying $\Lambda \subset \bar{\Lambda} \subset \tilde{\Lambda}$, $\norm{ \Restr_{\bar{\Lambda}} \br_{\Lambda} } \geq \alpha \norm{ \br_{\Lambda} }$
    \end{flushleft}
	\begin{algorithmic}[1]
		\State compute and rearrange contractions of new indices as specified in \eqref{contractions_new}
		\State find the minimal $M \in \mathbb{N}$ satisfying \eqref{update_rule_contr}.
		\State construct $\Lambda^{\mathrm{min}}(\br_{\Lambda};M)$ as \eqref{set} and compute $\norm{ \Restr_{\Lambda^{\mathrm{min}}(\br_{\Lambda};M)} (\br_{\Lambda}) }$
		\If{$\norm{ \Restr_{\Lambda^{\mathrm{min}}(\br_{\Lambda};M)} (\br_{\Lambda}) }  \geq \alpha \norm{ \br_\Lambda }$} \label{alg_expand_ifcond}
			\State $\bar{\Lambda} \leftarrow \Lambda^{\mathrm{min}}(\br_{\Lambda};M)$
		\Else
            \State $\bar{\Lambda}_i \leftarrow \tilde{\Lambda}_i, i = 1,\dots, J$ 
            \State rearrange contractions in \eqref{contractions_sparse} and find the minimal $K \in \mathbb{N}$ satisfying \eqref{update_rule_sparse}
            \State $\bar{\Lambda}_0 \leftarrow \Lambda_0 \cup \{(\bar\nu^j, \lambda^j),\   j = 1,\dots,K\}$
        \EndIf
		\State return $\bar{\Lambda}$
	\end{algorithmic}
\end{algorithm}

\begin{theorem}
    \label{thm:bound_bar_Lambda}
    For a fixed $\alpha \in (0,1)$, let $\widehat{\Lambda} = \widehat{\Lambda}_0 \times \cdots\times \widehat{\Lambda}_J \subset \cG$ be such that $\sum_{j=0}^J \#(\widehat{\Lambda}_j \backslash \Lambda_j)$ is minimal among such index sets satisfying \eqref{update_rule_sect4}.    
   Then $\bar{\Lambda}$, generated by Algorithm 3, satisfies $\norm{ \Restr_{\bar{\Lambda}} \br_{\Lambda} } \geq \alpha \norm{ \br_{\Lambda} }$ and has the following property:
\begin{equation}
    \label{set_size_ineq}
    \sum_{j=0}^J \#(\bar{\Lambda}_j \backslash \Lambda_j) \leq \sum_{j=0}^J \#(\widehat{\Lambda}_j \backslash \Lambda_j) + J .
\end{equation}
\end{theorem}
\begin{proof}
    On the one hand, if the condition in line 5 is satisfied, then the resulting set $\bar\Lambda$ satisfies \eqref{update_rule_sect4}. With
    \[
        M = \sum_{j=0}^J \#(\bar{\Lambda}_j \backslash \Lambda_j), \quad M^* = \sum_{j=0}^J \#(\widehat{\Lambda}_j \backslash \Lambda_j),
    \]
we show that $M = M^*$, implying \eqref{set_size_ineq} in this case. Since $\widehat{\Lambda}$ is the optimal product set satisfying \eqref{update_rule_sect4}, it follows that $M \geq M^*$. To show that the strict inequality cannot hold, we proceed by contradiction. Assuming that $M > M^*$, and since $M$ is the smallest integer satisfying \eqref{update_rule_contr}, the following inequality holds 
    \[
        \norm{ \Restr_{\widehat{\Lambda}} \br_{\Lambda} }^2 \leq \norm{ \Restr_\Lambda \br_\Lambda }^2 + \sum_{n=1}^{M^*} \abs{ \pi_{n}^{*}}^2 < \alpha^2 \norm{ \br_{\Lambda} }^2.
    \]
    This leads to a contradiction with the condition \eqref{update_rule_sect4} on $\widehat{\Lambda}$.
    
    On the other hand, if the condition in line \ref{alg_expand_ifcond} is not satisfied, we first demonstrate that $\bar{\Lambda}$ satisfies $\| \Restr_{\bar{\Lambda}} \br_{\Lambda} \| \geq \alpha \| \br_{\Lambda}\|$. 
    By Proposition \ref{prop:contractions}(i),
    \[ \|\br_{\Lambda}\|^2 = \sum_{n=1}^\infty \abs{ \pi_{n}^{**}}^2 + \sum_{(\bar\nu,\lambda) \in \Lambda_0} |\pi^{(0)}_{(\bar\nu,\lambda)}(\br_\Lambda)|^2
    \]
    Since $\bar{\Lambda}_i = \tilde{\Lambda}_i, i = 1,\dots,J$ with $\tilde{\Lambda}_i$ being the largest possible set, and $K$ as defined in \eqref{update_rule_sparse}, it follows that 
    \[
        \norm{ \Restr_{\bar{\Lambda}} \br_{\Lambda} }^2 =  \sum_{n=1}^{K} \abs{ \pi_{n}^{**}}^2 + \sum_{(\bar\nu,\lambda) \in \Lambda_0} \abs{ \pi^{(0)}_{(\bar\nu,\lambda)}(\br_\Lambda) }^2 \geq \alpha^2 \norm{ \br_{\Lambda} }^2.
    \]
    Next, we prove \eqref{set_size_ineq}, namely 
    \[
        \sum_{j=0}^J \#(\bar{\Lambda}_j \backslash \Lambda_j) = K + J \leq \sum_{j=0}^J \#(\widehat{\Lambda}_j \backslash \Lambda_j) + J = M^* + J,
        \]
    which is equivalent to proving that $K \leq  M^*$. Again, we proceed by contradiction and assume that $K > M^*$. Since $K$ is the smallest integer satisfying \eqref{update_rule_sparse}, we have 
    \[
        \norm{ \Restr_{\widehat{\Lambda}} \br_{\Lambda} }^2 \leq \sum_{n=1}^{M^*} \abs{ \pi_{n}^{**}}^2 + \sum_{(\bar\nu,\lambda) \in \Lambda_0} |\pi^{(0)}_{(\bar\nu,\lambda)}(\br_\Lambda)|^2 < \alpha^2 \|\br_{\Lambda}\|^2,
    \]
    which contradicts the definition of $\widehat{\Lambda}$.
\end{proof}

The proof of the following result combines ideas from \cite[Lemma~2.1]{GantumurHarbrechtStevenson:07} and \cite[Theorem~7]{BachmayrDahmen:14}. 

\begin{theorem}
    \label{thm:optimal_set_size}
    Let $\alpha \in (0,\kappa(\bA)^{-\frac{1}{2}})$ and $\Lambda = \Lambda_0\times \cdots \times \Lambda_J \subset \cG$ be a finite product set. Let $\bw_{\Lambda} \in \ell_2(\cG)$ with $\supp \bw_\Lambda \subseteq \Lambda$. We assume that $\pi^{(0)}(\bu) \in \mathcal{A}^{s_1}$, $\pi^{(j)}(\bu) \in \mathcal{A}^{s_2}$ for $j=1,\dots,J$. Let $\Lambda^* = \Lambda_0^* \times \cdots \times \Lambda_J^* \subset \cG$ be a minimizer of $\sum_{j=0}^J \#(\Lambda_j^* \backslash \Lambda_j)$ among the product sets such that
    \begin{equation}
        \label{eq:update_rule_exact_resi}
        \norm{\Restr_{\Lambda^*}(\bA \bw_\Lambda -\bbf)} \geq \alpha \norm{\bA \bw_\Lambda - \bbf} .
    \end{equation}
    Then 
    \begin{multline}
        \label{set_size_bound}
        \sum_{j=0}^J \#(\Lambda_j^* \backslash \Lambda_j)  \leq 2C^{\frac{1}{s_1}} \norm{\bA \bw_{\Lambda} - \bbf}^{-\frac{1}{s_1}} \norm{\pi^{(0)}(\bu)}_{\cA^{s_1}}^{\frac{1}{s_1}}  \\[-9pt]  
          + 2 C^{\frac{1}{s_2}} \norm{\bA \bw_{\Lambda} - \bbf}^{-\frac{1}{s_2}} \sum_{j=1}^{J} \norm{\pi^{(j)}(\bu)}_{\cA^{s_2}}^{\frac{1}{s_2}},
    \end{multline}
    where $C = (J+1) \sqrt{\norm{\bA}} / \lambda $.
\end{theorem}
\begin{proof}
Let $\lambda > 0$ be a constant such that $\alpha \leq \kappa(\bA)^{-\frac{1}{2}}\left(1- \|\bA\|\lambda^2\right)^{\frac{1}{2}}$. Since $\alpha \in (0,\kappa(\bA)^{-\frac{1}{2}})$, it requires that $0< \left(1-\|\bA\|\lambda^2\right)^{\frac{1}{2}}<1$, which implies that $0< \lambda < \|\bA\|^{-\frac{1}{2}}$. 

    For each $N \in \N$ and each $\bv \in \ell_2(\cG)$, there exists $\Lambda(\bv;N) = \Lambda_0(\bv;N) \times \cdots \times \Lambda_J(\bv;N) \subset \cG$ satisfying $\sum_{j=0}^J \# \Lambda_j(\bv;N) = N$, such that
    \begin{equation}
        \label{eq:optimal_coarsening}
        \| \bv - \Restr_{\Lambda(\bv;N)}\bv \| = \min \biggl\{ \|\bv-\mathbf{w}\| \colon \sum_{j=0}^J\#\supp_j(\mathbf{w})\leq N \biggr\} .
    \end{equation}
    Let $M \in \N$ be the smallest integer such that $\| \bu - \Restr_{\Lambda(\bu;M)}\bu \| \leq \lambda\|\bu-\bw_{\Lambda} \|_\bA$. For each $j\in \{0,\ldots,J\}$, it holds that $\Lambda_j(\bu;M) \neq \emptyset$, which implies $M \geq J+1$. Otherwise, we have $\Lambda(\bu;M) = \emptyset$, which leads to
    \[
        \norm{\bu} = \| \bu - \Restr_{\Lambda(\bu;M)}\bu \| \leq \lambda\|\bu-\bw_{\Lambda} \|_\bA \leq \lambda \norm{\bA}^{\frac{1}{2}} \norm{\bu-\bw_{\Lambda}} < \norm{\bu},
    \]
    where the last inequality follows from $\lambda < \norm{\bA}^{-\frac{1}{2}}$ and it is a contradiction.

    By the definition of $M$ and the optimality \eqref{eq:optimal_coarsening}, we have
    \begin{align*}
    \lambda\|\bu-&\bw_{\Lambda} \|_\bA  <  \|\bu-\Restr_{\Lambda(\bu;M-1)}\bu\| \\ 
    & \leq \inf \biggl\{ \bignorm{\pi^{(0)}(\bu)-\Restr_{\Lambda_0} \pi^{(0)}(\bu)} + \sum_{j=1}^J \bignorm{\pi^{(j)}(\bu)-\Restr_{\Lambda_j} \pi^{(j)}(\bu)} \; \colon \sum_{j=0}^{J} \# \Lambda_j \leq M-1 \biggr\}  \,,
      \end{align*}
    where
    \begin{multline*}
      \bignorm{\pi^{(0)}(\bu)-\Restr_{\Lambda_0} \pi^{(0)}(\bu)} + \sum_{j=1}^J \bignorm{\pi^{(j)}(\bu)-\Restr_{\Lambda_j} \pi^{(j)}(\bu)}  \\[-6pt]
        \leq 
        (\# \Lambda_0)^{-s_1} \bignorm{\pi^{(0)}(\bu)}_{\cA^{s_1}} + \sum_{j=1}^{J} (\# \Lambda_j)^{-s_2} \bignorm{\pi^{(j)}(\bu)}_{\cA^{s_2}} .
    \end{multline*}
    To obtain an upper bound for $M$, we would like to find index sets $\Lambda_0,\ldots,\Lambda_J$ with minimal $\sum_{j=0}^{J} \# \Lambda_i $ such that 
    \[
    \lambda\|\bu-\bw_{\Lambda} \|_\bA \geq (\# \Lambda_0)^{-s_1} \|\pi^{(0)}(\bu)\|_{\cA^{s_1}} + \sum_{j=1}^{J} (\# \Lambda_j)^{-s_2} \|\pi^{(j)}(\bu)\|_{\cA^{s_2}}
    \]
    to conclude that $M \leq \sum_{j=0}^{J} \# \Lambda_j $. Equilibrating the upper bound yields 
        \[
    \begin{gathered}
        \# \Lambda_0 = \left\lceil \left(\frac{(J+1)\|\pi^{(0)}(\bu)\|_{\cA^{s_1}}} {\lambda \|\bu-\bw_{\Lambda} \|_\bA}\right)^{\frac{1}{s_1}} \right\rceil, \\
        \# \Lambda_i = \left\lceil \left(\frac{(J+1) \|\pi^{(j)}(\bu)\|_{\cA^{s_2}}} {\lambda \|\bu-\bw_{\Lambda} \|_\bA}\right)^{\frac{1}{s_2}} \right\rceil, \quad j=1,\dots,J.
    \end{gathered}
    \]
    Given $M \geq J+1$, we note that $\lambda\norm{\bu-\bw_{\Lambda}}_\bA < \norm{\bu}$. By definition of $\norm{\cdot}_{\cA^s}$ given in \eqref{eq:approximation_space}, it holds that $\norm{\bu} \leq \norm{\pi^{(0)}(\bu)}_{\cA^{s_1}}$, and $ \norm{\bu}\leq \norm{\pi^{(j)}(\bu)}_{\cA^{s_2}}$ for $j= 1,\ldots,J$. This implies 
    \[
        \left(\frac{(J+1)\norm{\pi^{(0)}(\bu)}_{\cA^{s_1}}} {\lambda \norm{\bu-\bw_{\Lambda}}_\bA}\right)^{\frac{1}{s_1}} > 1, \quad \left(\frac{(J+1)\norm{\pi^{(j)}(\bu)}_{\cA^{s_2}}} {\lambda \norm{\bu-\bw_{\Lambda}}_\bA}\right)^{\frac{1}{s_2}} > 1.
    \]
    With the above estimates, we have
    \begin{align*}
        M & \leq \left\lceil \left(\frac{(J+1)\norm{\pi^{(0)}(\bu)}_{\cA^{s_1}}} {\lambda \norm{\bu-\bw_{\Lambda}}_\bA}\right)^{\frac{1}{s_1}} \right\rceil + \sum_{j=1}^{J} \left\lceil \left(\frac{(J+1)\norm{\pi^{(j)}(\bu)}_{\cA^{s_2}}} {\lambda \norm{\bu-\bw_{\Lambda}}_\bA}\right)^{\frac{1}{s_2}} \right\rceil \\
        & < 2\left(\frac{(J+1)\norm{\pi^{(0)}(\bu)}_{\cA^{s_1}}} {\lambda \norm{\bu-\bw_{\Lambda}}_\bA}\right)^{\frac{1}{s_1}} + 2\sum_{j=1}^{J} \left(\frac{(J+1)\norm{\pi^{(j)}(\bu)}_{\cA^{s_2}}} {\lambda \norm{\bu-\bw_{\Lambda}}_\bA}\right)^{\frac{1}{s_2}} 
    \end{align*}
    Together with $\|\bu-\bw_{\Lambda} \|_\bA \geq \|\bA\|^{-\frac{1}{2}}\|\bA \bw_{\Lambda} - \bbf\|$, we obtain
    \[
        M \leq 2C^{\frac{1}{s_1}} \norm{\bA \bw_{\Lambda} - \bbf}^{-\frac{1}{s_1}} \norm{\pi^{(0)}(\bu)}_{\cA^{s_1}}^{\frac{1}{s_1}} + 2 C^{\frac{1}{s_2}} \norm{\bA \bw_{\Lambda} - \bbf}^{-\frac{1}{s_2}} \sum_{j=1}^{J} \norm{\pi^{(j)}(\bu)}_{\cA^{s_2}}^{\frac{1}{s_2}},
    \]
    where $C = (J+1) \sqrt{\norm{\bA}} / \lambda$.

    Let $\hat{\Lambda} = \Lambda \cup  \Lambda(\bu;M) $, then the Galerkin solution $\bu_{\hat{\Lambda}}$ satisfies
    \[
      \|\bu - \bu_{\hat{\Lambda}}\|_\bA  \leq \| \bu -  \Restr_{\Lambda(\bu;M)}\bu \|_\bA \leq \|\bA\|^{\frac{1}{2}}\| \bu -  \Restr_{\Lambda(\bu;M)}\bu \| \leq \|\bA\|^{\frac{1}{2}}\lambda \|\bu-\bw_{\Lambda} \|_\bA .
    \]
    By Galerkin orthogonality $\| \bu_{\hat{\Lambda}} - \bw_{\Lambda} \|_\bA \geq (1-\|\bA\|\lambda^2)^{\frac{1}{2}}\|\bu - \bw_{\Lambda}\|_\bA$, and thus
    \[
        \begin{aligned}
            \|\Restr_{\hat{\Lambda}}(\bA \bw_{\Lambda} - \bbf)\| & = \| \Restr_{\hat{\Lambda}}(\bA \bu_{\hat{\Lambda}} - \bA \bw_{\Lambda})\| \geq \|\bA^{-1}\|^{-\frac{1}{2}}\|\bu_{\hat{\Lambda}}-\bw_{\Lambda}\|_\bA \\
            & \geq \|\bA^{-1}\|^{-\frac{1}{2}}(1-\|\bA\|\lambda^2)^{\frac{1}{2}}\|\bu - \bw_{\Lambda}\|_\bA \\
            & \geq \kappa(\bA)^{-\frac{1}{2}}(1-\|\bA\|\lambda^2)^{\frac{1}{2}}\|\bA \bw_{\Lambda} - \bbf\| \\
            & \geq \alpha\|\bA \bw_{\Lambda} - \bbf\|.
        \end{aligned} 
    \]
    Given that $\Lambda^*$ is the optimal set satisfying \eqref{eq:update_rule_exact_resi} and $\Lambda \subset \hat{\Lambda}$, we conclude that
    \[
        \sum_{j=0}^J \#(\Lambda_j^* \backslash \Lambda_j)  \leq \sum_{j=0}^J\#(\hat{\Lambda}_j\backslash \Lambda_j)   \leq M ,
    \]
    which is the inequality in \eqref{set_size_bound}.
\end{proof}

\subsection{Complexity}

Based on Theorems \ref{thm:bound_bar_Lambda} and \ref{thm:optimal_set_size}, we show that the adaptive solver proposed in Algorithm \ref{alg_adaptive} yields quasi-optimal discretizations without re-coarsening in Theorem \ref{thm:complexity}.

\begin{theorem}
    \label{thm:complexity}
    Let $\bu \in \ell_2(\cG)$ be the solution to equation \eqref{operator-equation}. We assume $\pi^{(0)}(\bu) \in \cA^{s_1}$, $\pi^{(j)}(\bu) \in \cA^{s_2}$ for $j=1,\dots,J$. Let $\alpha \in (0, \kappa(\bA)^{-\frac{1}{2}})$ be the bulk chasing parameter in line \ref{alg:Adap_expand}, $\mu= \frac{\alpha+\delta}{1-\delta}$, and choose the relative residual approximation tolerance $\delta \in (0,\alpha)$ in line \ref{alg:Adap_resi_relative_error} of Algorithm \ref{alg_adaptive}, such that $\mu < \kappa(\bA)^{-\frac{1}{2}}$. Let $0<\eta < \frac{(1-\delta)(\alpha-\delta)}{1+\delta}\kappa(\bA)^{-1}$ be the relative approximate Galerkin solution tolerance in line \ref{alg:Adap_solve_relative_error}. Then Algorithm \ref{alg_adaptive} produces an approximation $\bu_{\varepsilon}$ to $\bu$ with finite support $\supp\bu_{\varepsilon} \subseteq \Lambda$ in finitely many steps such that $\|\bA \bu_{\varepsilon} - \bbf\| \leq \varepsilon$
    and
    \begin{equation}
        \label{support_bound}
        \sum_{j=0}^J \#\Lambda_j \lesssim  \varepsilon^{-\frac{1}{s_1}}\norm{\pi^{(0)}(\bu)}_{\cA^{s_1}}^{\frac{1}{s_1}} + \varepsilon^{-\frac{1}{s_2}} \biggl( \sum_{j=1}^{J} \norm{\pi^{(j)}(\bu)}_{\cA^{s_2}}^{\frac{1}{s_2}} \biggr).
    \end{equation}
Moreover, with $\varepsilon_k$ as in Algorithm \ref{alg_adaptive},
    \begin{enumerate}[{\rm(i)}]
       \item If there exist $M>0$ and $0<p<2$ such that $\sigma_{t,j}(\bu_{\Lambda^k}) \leq M j^{-\frac1p}$ for $j \in \N$, $t=1,\dots,E$ for each $k$, we have
            \begin{equation}\label{rank_bound_1}
                \| \bu - \bu_k   \| \lesssim (J+1)\varepsilon_k, \quad
                \max\limits_{t=1,\dots,E}\rank_t (\bu_k) \lesssim (J+1)^2 M^{\frac{1}{s}} \varepsilon_k^{-\frac{1}{s}}, \quad s = \frac{1}{p}-\frac{1}{2}.
            \end{equation}
        \item If there exist $C,c,\beta>0$ such that $\sigma_{t,j}(\bu_{\Lambda^k}) \leq Ce^{-cj^\beta}$, for $j\in \N$, $t=1,\dots,E$ for each $k$, we have
            \begin{equation}\label{rank_bound_2}
                \|\bu - \bu_k \| \lesssim (J+1)\varepsilon_k, \quad
                \max\limits_{t=1,\dots,E}\rank_t (\bu_k ) \lesssim (J+1)^2 \bigl(1+\abs{ \ln\varepsilon_k} \bigr)^{\frac{1}{\beta}}
            \end{equation}
    \end{enumerate} 
    In both cases, the constants depend on the quantities specified in Theorem \ref{thm:stsolve}.
\end{theorem}
\begin{proof}
    Since the parameters satisfy $0< \delta < \alpha$, $\mu= \frac{\alpha+\delta}{1-\delta} < \kappa(\bA)^{-\frac{1}{2}}$ and 
    \[0<\eta < \frac{(1-\delta)(\alpha-\delta)}{1+\delta}\kappa(\bA)^{-1},\] 
    we directly obtain from \cite[Prop.~4.2]{Stevenson:09} that
    \begin{equation}
        \label{error_reduction}
        \|\bu-\bu_{k+1}\|_\bA \leq \rho \|\bu-\bu_k\|_\bA 
    \end{equation}
    with \[
    \rho=\left[1-\left(\frac{\alpha-\delta}{1+\delta}\right)^2 \kappa(\bA)^{-1}+\frac{\eta^2}{(1-\delta)^2} \kappa(\bA)\right]^{\frac{1}{2}} < 1.\]
    Thus Algorithm \ref{alg_adaptive} terminates in at most $K$ steps, where $K$ is the minimal integer such that $\rho^K \Gamma^{1/2} \|\bu\|_\bA \leq \varepsilon$, so that $\bu_{\varepsilon}=\bu_K$. 

    Let $\bu_k$ be the current iterate and $\br_k$ be the  current approximate residual of Algorithm \ref{alg_adaptive} at step $k$. We illustrate that $\Lambda^{k+1}$ produced in line \ref{alg:Adap_expand} satisfies the following bound
    \[
        \sum_{j=0}^J \#(\Lambda^{k+1}_j\backslash \Lambda^k_j) \leq 3C^{\frac{1}{s_1}} \norm{\bA \bu_{k} - \bbf}^{-\frac{1}{s_1}} \norm{\pi^{(0)}(\bu)}_{\cA^{s_1}}^{\frac{1}{s_1}} + 3 C^{\frac{1}{s_2}} \norm{\bA \bu_{k} - \bbf}^{-\frac{1}{s_2}} \sum_{j=1}^{J} \norm{\pi^{(j)}(\bu)}_{\cA^{s_2}}^{\frac{1}{s_2}},
    \]
    where $C = (J+1) \sqrt{\norm{\bA}} / \lambda$.
    
    Let $\widehat{\Lambda} = \widehat{\Lambda}_0 \times \cdots\times \widehat{\Lambda}_J \subset \cG$ be such that $\sum_{j=0}^J \#(\widehat{\Lambda}_j \backslash \Lambda_j)$ is minimal among index sets satisfying
    \[
        \norm{\Restr_{\widehat{\Lambda}}\br_k} \geq \alpha \norm{\br_k}.
    \]
    Let $\lambda > 0$ be a constant such that $\mu \leq \kappa(\bA)^{-\frac{1}{2}}(1- \|\bA\|\lambda^2)^{\frac{1}{2}}$ and $\Lambda^* = \Lambda_0^* \times \cdots \times \Lambda_J^* \subset \cG$ be a minimizer of $\sum_{j=0}^J \#(\Lambda_j^* \backslash \Lambda_j)$ among the product sets such that
    \[
        \norm{\Restr_{\Lambda^*}(\bA \bu_k -\bbf)} \geq \mu \norm{\bA \bu_k - \bbf} .
    \]
    Since $\norm{\br_k -(\bA \bu_k - \bbf)} \leq \delta\norm{\br_k}$ and by definition of $\Lambda^*$, we note that $\mu \norm{\br_k} \leq \mu \norm{\bA \bu_k - \bbf} + \mu \delta \norm{\br_k} \leq \norm{\Restr_{\Lambda^*}(\bA \bu_k -\bbf)} + \mu \delta \norm{\br_k} \leq \norm{\Restr_{\Lambda^*} \br_k} + (\delta + \mu \delta) \norm{\br_k}$. Namely, we have $\norm{\Restr_{\Lambda^*} \br_k} \geq (\mu - \delta -\mu\delta)\norm{\br_k} = \alpha \norm{\br_k}$. With definition of $\widehat{\Lambda}$, we conclude that
    \[
        \sum_{j=0}^J \#(\widehat{\Lambda}_j \backslash \Lambda_j) \leq \sum_{j=0}^J \#(\Lambda_j^* \backslash \Lambda_j) .
    \]
    Together with estimates \eqref{set_size_ineq} and \eqref{set_size_bound}, and noting that all summands in the bound of \eqref{set_size_bound} are greater than $1$, we increase the multiplier $2$ to $3$ to eliminate the summation of $J$. Thus, we obtain
    \[
    \begin{aligned}
        \sum_{j=0}^J \#(\Lambda^{k+1}_j\backslash & \Lambda^k_j) \leq \sum_{j=0}^J \#(\Lambda_j^* \backslash \Lambda_j) + J \\
        & \leq 3C^{\frac{1}{s_1}} \norm{\bA \bu_{k} - \bbf}^{-\frac{1}{s_1}} \norm{\pi^{(0)}(\bu)}_{\cA^{s_1}}^{\frac{1}{s_1}} + 3 C^{\frac{1}{s_2}} \norm{\bA \bu_{k} - \bbf}^{-\frac{1}{s_2}} \sum_{j=1}^{J} \norm{\pi^{(j)}(\bu)}_{\cA^{s_2}}^{\frac{1}{s_2}} 
    \end{aligned}
    \]

    Through recursive summation, we obtain
    \[
    \begin{aligned}
        \sum_{j=0}^J \#(\Lambda^{K}_j) & = \sum_{k=0}^{K-1} \sum_{j=0}^J \#(\Lambda^{k+1}_j\backslash \Lambda^k_j)  \\
        & \leq 3C^{\frac{1}{s_1}} \norm{\pi^{(0)}(\bu)}_{\cA^{s_1}}^{\frac{1}{s_1}} \biggl(\sum_{k=0}^{K-1} \norm{\bA \bu_k - \bbf}^{-\frac{1}{s_1}}\biggr) \\
        & \qquad + 3 C^{\frac{1}{s_2}} \biggl(\sum_{j=1}^{J}\norm{\pi^{(j)}(\bu)}_{\cA^{s_2}}^{\frac{1}{s_2}} \biggr) \biggl(\sum_{k=0}^{K-1} \norm{\bA \bu_k - \bbf}^{-\frac{1}{s_2}}\biggr). \\
    \end{aligned}
    \]
    For the factors concerning the sum of residuals and for $ s_i \in \{s_1,s_2\}$, we have the estimate
    \[
    \begin{aligned}
        \biggl(\sum_{k=0}^{K-1} \norm{\bA \bu_k - \bbf}^{-\frac{1}{s_i}}\biggr) & \leq  \biggl(\sum_{k=0}^{K-1} \norm{\bA^{-1}}^{\frac{1}{2s_i}}\norm{\bu-\bu_k}_\bA^{-\frac{1}{s_i}} \biggr) \\
        & \leq \norm{\bA^{-1}}^{\frac{1}{2s_i}} \biggl(\sum_{k=0}^{K-1} (\rho^{K-1-k})^{\frac{1}{s_i}}\norm{\bu-\bu_{K-1}}_\bA^{-\frac{1}{s_i}}\biggr) \\
        & \leq \kappa(\bA)^{\frac{1}{2s_i}}\biggl(\sum_{k=0}^{K-1} (\rho^{K-1-k})^{\frac{1}{s_i}}\norm{\bA \bu-\bA \bu_{K-1}}^{-\frac{1}{s_i}}\biggr) \\
        & \leq  \frac{\kappa(\bA)^{\frac{1}{2s_i}}}{1-\rho^{\frac{1}{s_i}}}\varepsilon^{-\frac{1}{s_i}},
    \end{aligned}
    \]
    which altogether results in
    \[
        \sum_{j=0}^J \#(\Lambda_j) \leq   \frac{3\bar{C}^{\frac{1}{s_1}}}{1-\rho^{\frac{1}{s_1}}}\varepsilon^{-\frac{1}{s_1}}\norm{\pi^{(0)}(\bu)}_{\cA^{s_1}}^{\frac{1}{s_1}} +   \frac{3 \bar{C}^{\frac{1}{s_2}}}{1-\rho^{\frac{1}{s_2}}}\varepsilon^{-\frac{1}{s_2}} \biggl( \sum_{j=1}^{J} \norm{\pi^{(j)}(\bu)}_{\cA^{s_2}}^{\frac{1}{s_2}} \biggr),
    \]
    where $\bar{C} = (J+1) \sqrt{\kappa(A)\norm{\bA}} / \lambda$.

    Since $\bu_{k}$ is the approximate Galerkin solution with error bound $\eta \|\br_{k-1}\| \geq \eta \|\br_{k}\| \eqsim \eta \varepsilon_k$,
    the rank estimates \eqref{rank_bound_1} and \eqref{rank_bound_2} follow directly from Theorem \ref{thm:stsolve}.
\end{proof}

\begin{remark}
  The result of Theorem \ref{thm:complexity} improves on those of \cite{BachmayrDahmen:14,BachmayrCohenDahmen:18}, which are based on approximate Richardson iteration, in two main points. First, no coarsening of discretizations is required to maintain quasi-optimality. Here, exploiting the particular structure of the parametric problems plays an important role. 
  
 Second, the rank bounds in Theorem \ref{thm:stsolve} and in \eqref{rank_bound_1} and \eqref{rank_bound_2}, which are satisfied by all generated approximate solutions, lead to substantially more favorable bounds on the ranks of intermediate quantities. This results from the quasi-optimal rank bounds for each single iterate achieved by soft thresholding, in place of recurrent low-rank recompressions after several steps as in \cite{BachmayrDahmen:14}. The largest intermediate ranks that can occur in our case are thus those of residuals. Since the action of $\bA$ increases hierarchical ranks by at most a factor $J+1$, and assuming for simplicity that $\bbf$ has fixed finite ranks, we also have $\rank_t (\bA \bu_k - \bbf) \lesssim (J+1) \rank_t (\bu_k ) + \rank_t (\bbf)$, and analogous bounds hold for all further ranks that arise during the iteration.
 
However, we need to make a slightly stronger assumption on the decay of singular values $\sigma_t(\bu_{\Lambda^k})$, $t = 1,\ldots, E$, of the arising Galerkin discretizations, whereas in \cite{BachmayrDahmen:14}, only the decay of $\sigma_t(\bu)$ is used. As we have noted, although we have the componentwise bound $\sigma_t( \Restr_{\Lambda^k}\bu) \leq \sigma_t(\bu)$ for best approximations on $\Lambda^k$, this bound does not directly transfer to Galerkin solutions $\bu_{\Lambda^k}$.
\end{remark}

\begin{remark}
In their basic form, the solver in Algorithm \ref{alg_GalerkinSolver} and the discretization refinement in Algorithm \ref{alg_adaptive} both depend on $\mathbf A$ being well-conditioned: the condition number of $\mathbf A$ enters in the number of required operations in Algorithm \ref{alg_GalerkinSolver} and especially also in the choice of bulk chasing parameter $\alpha$ in Algorithm \ref{alg_adaptive} that influences the number of degrees of freedom that can be added in each refinement step. As a consequence, to maintain efficiency in case of problems with small values of $r$, corresponding to coefficients with high contrast, additional preconditioning may be required.

In the rank bounds \eqref{rank_bound_1} and \eqref{rank_bound_2}, we obtain a fixed low-order polynomial dependence on $J$. While this could be expected to be practically limiting in quantitative terms for very large $J$, in our numerical tests, the approximation ranks of the generated solutions are not observed to carry such additional factors -- that is, these ranks are determined rather directly by the decay of singular values of matricizations. We thus conjecture that these $J$-dependent prefactors can be reduced or removed from the bounds. In addition, there may be a dependence of the number of iterations in Algorithm \ref{alg_GalerkinSolver} on $\log J$ as noted in Remark \ref{rem:logJiter}.
\end{remark}

\section{Numerical experiments}\label{sec:numtests}

In this section, we choose two examples to illustrate the performance of the proposed method. In particular, we are interested in the evolution of ranks and discretizations. We consider expansions of $a$ with a wavelet-like multilevel structure and with a Karhunen-Lo\`eve-type expansion in terms of globally supported trigonometric functions. The Julia code used for these numerical experiments is available in the Zenodo repository \cite{zenodo}. 

\subsection{Wavelet-type coefficients.}

We first consider that the random field $a(x,y)$ exhibits an isotropic dependence on parameters within set $\mathcal{I}_1$, while it shows anisotropic dependence on parameters in set $\mathcal{I}_2$ with expansion in terms of hierarchical hat functions. This provides a simplified model for wavelet-type expansions of random fields, as developed for Mat\'ern random fields in \cite{BCM:18}. Such parameterizations are a favorable choice for adaptive solvers especially in view of \eqref{fully_discrete_approximation_rate}. The parameterized coefficient is given by
\begin{equation}\label{eq:wavelettype}
    a(x,y) = 1 + c_1 \sum_{i \in \mathcal{I}_1} \chi_{D_i}(x)y_i + c_2\sum_{(\ell,k) \in \mathcal{I}_2 } \theta_{\ell,k}(x)y_{\ell,k}, \quad D = (0,1),
\end{equation}
where $c_1, c_2 \in (0,1)$ are constants satisfying \eqref{UEA} and $\{D_i \colon i \in \cI_1\}$ with $\cI_1=\{1,2,\dots,N\}$ is a uniform partition of $D$ into $N$ subintervals for spatial dimension $d=1$. The functions $\theta_{\ell,k}(x)$ are defined as
\begin{equation}
    \theta_{\ell,k}(x) = 2^{-\alpha \ell}\theta(2^\ell x - k), \quad  \cI_2 = \{(\ell,k) : k = 0,1,\dots,2^{\ell}-1, \, \ceil{\log_2 N} \leq \ell \leq L\},
\end{equation}
with $\theta(x) = \max\{0,1-|2x -1|\}$. Here we set $\ell \geq \ceil{\log_2 N}$ to ensure that the support size of each $\theta_{\ell,k}$ does not exceed that of subinterval. The source term is chosen as
\begin{equation}
    f(x) = \operatorname{sgn}\bigl(\tfrac12-x\bigr) \bigabs{\tfrac12-x}^{-1/4}. 
\end{equation}
In Figure \ref{plot:short_long_correlation}, we display two diffusion coefficients derived from two different random fields of the type \eqref{eq:wavelettype}, which exhibit short and long correlation lengths, respectively. As can be observed, the red curve exhibits significant oscillations over a small length scale and indeed needs a finer spatial discretization compared to that required for the diffusion coefficient with a long correlation length, as shown by the blue curve.

\begin{figure}[htbp]
    \centering
    \includegraphics{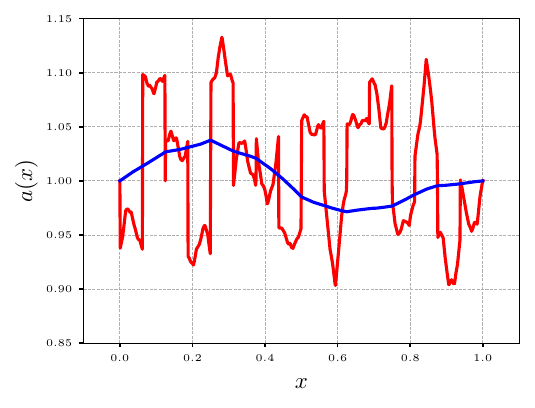}
    \caption{Realizations of two different random fields. \textit{Red line:} a diffusion coefficient with short correlation length, and \textit{blue line:} a diffusion coefficient with long correlation length.}
    \label{plot:short_long_correlation}
\end{figure}

The adaptive Galerkin scheme is tested with $N =8$, $L = 5$ and $\alpha = 2$. This configuration results in $8$ dominant parameters, as well as $56$ parameters exhibiting decreasing dependence. We compare our new tensor approximation format with sparse expansion, as in \eqref{eq:lrstandardform} and \eqref{HT-Sparse-coefficients}, to a standard tensor approximation separating all 65 variables. Such a standard low-rank approximation here amounts to taking $J=64$, which is the total number of parametric variables. In what follows, the results for new format are labeled as $J=8$, while the results for standard format are accordingly labeled as $J=64$. The reference curve labelled $J=1$, where only one parameter is kept in a separate tensor mode, providing a sparse approximation-dominated reference within the same implementation framework.

We study two contrast regimes: with $c_1=0.8,c_2=3.2$ leading to condition numbers $\kappa(\bA) \approx 133.55$ in the high-contrast case, and with $c_1=0.1,c_2=0.4$ yielding $\kappa(\bA) \approx 11.85$ in the low-contrast case.

In the left plots of Figure \ref{plot:wavelet_set_size_resi_high} and \ref{plot:wavelet_set_size_resi_low}, we illustrate the convergence with respect to the cardinality $\#\Lambda_0$ of the sparse expansion mode frame in terms of the residual error bound $\varepsilon$ for high and low contrast cases, respectively. The limiting approximation rate $\frac{4}{3}$ given by substituting $\alpha = 2$ in \eqref{fully_discrete_approximation_rate}, expected for sparse polynomial approximations, is recovered by the new format with $J=8$ (orange curve) in both cases. However, for the high-contrast case, the sparse-dominated approximation with $J=1$ (blue curve) reference does not reach the expected optimal rate in the observed range, while in the low-contrast case, it exhibits a visible plateau before the asymptotic decay becomes dominant.

Meanwhile, in the right plots of Figure \ref{plot:wavelet_set_size_resi_high} and Figure \ref{plot:wavelet_set_size_resi_low}, the average size of the parametric mode frames, $J^{-1}\bigl(\sum_{j=1}^{J}\#\Lambda_j\bigr)$, gives a complementary indication of whether the separated parametric variables are indeed dominant. For the new format with $J=8$, this average remains comparatively large, which shows that these eight separated variables require nontrivial polynomial discretizations and therefore have substantial influence on the solution. Conversely, when all parametric variables are put into separate modes, corresponding to $J=64$, most of the separated modes are tail variables with rapidly decreasing influence. Their associated polynomial frames remain small, so averaging over all $64$ modes leads to a much smaller value. This should therefore not be interpreted as an advantage of the fully separated format, but rather reflects dilution by many weak tail variables.

\begin{figure}[htbp]
    \centering
    \includegraphics{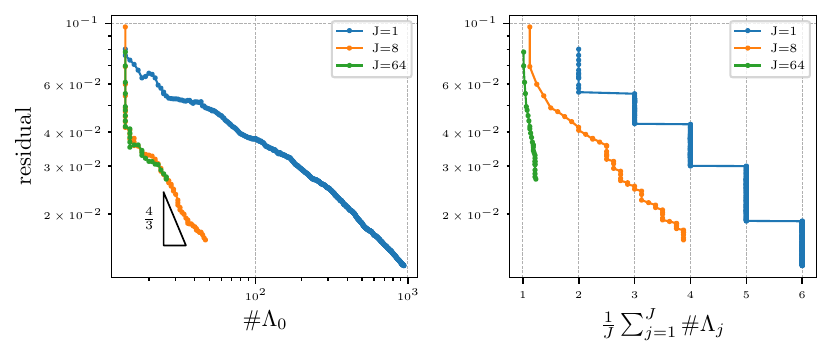}
    \caption{Computed residual bounds as a function of size of the sparse expansion mode frame (left) of iterates $\bu_k$, and average size of parametric mode frames (right) for high contrast ($\kappa(\bA) \approx133.55$).}
    \label{plot:wavelet_set_size_resi_high}
\end{figure}

\begin{figure}[htbp]
    \centering
    \includegraphics{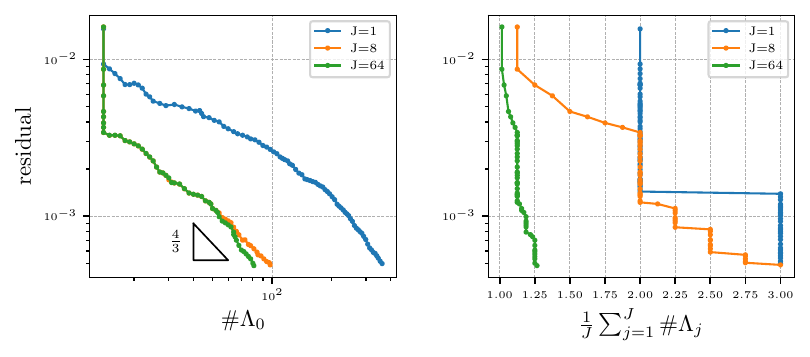}
    \caption{Computed residual bounds as a function of size of the sparse expansion mode frame (left) of iterates $\bu_k$, and average size of parametric mode frames (right) for low contrast ($\kappa(\bA) \approx11.85$).}
    \label{plot:wavelet_set_size_resi_low}
\end{figure}

The maximum hierarchical ranks of computed Galerkin iterates $\bu_k$ are shown in Figure \ref{plot:wavelet_rank}. It is observed that the ranks of $\bu_k$ in the new format exhibit a much slower increase during the iterations compared to the standard tensor format in both contrast regimes. Furthermore, while the ranks increase with $J$, their growth is more favorable than the quadratic dependence on $J$ as ensured by Theorem \ref{thm:complexity}.

\begin{figure}[htbp]
    \centering
    \includegraphics{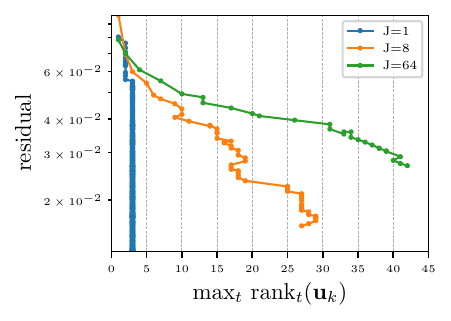}
    \hfill
    \includegraphics{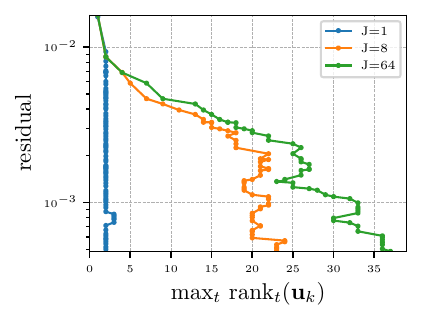}
    \caption{Maximum hierarchical ranks of iterates $\bu_k$, i.e., $\max_t \rank_t(\bu_k)$, plotted against the residual error bound $\varepsilon$, for different choices of $J$ and for high contrast (left) and low contrast (right).}
    \label{plot:wavelet_rank}
\end{figure}

The numerically observed decay of the hierarchical singular values for different formats are shown in Figure \ref{plot:wavelet_sv}. Note that in hierarchical tensor representations, the ranks of further matricizations are also included, as depicted by the dashed lines. For both contrast regimes, we observe exponential decay of singular values of matricization $\cM_0(\bu_\varepsilon)$ in both $J=8$ and $J=64$, and a faster decrease in the new format, which is consistent with the observed reduction in ranks.

\begin{figure}[htbp]
    \centering
    \includegraphics{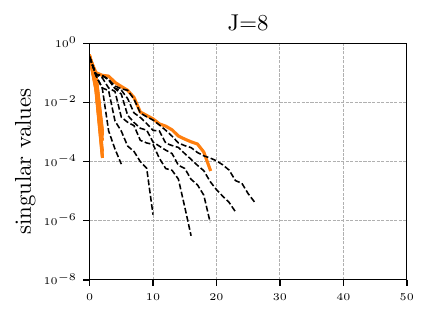}
    \hfill
    \includegraphics{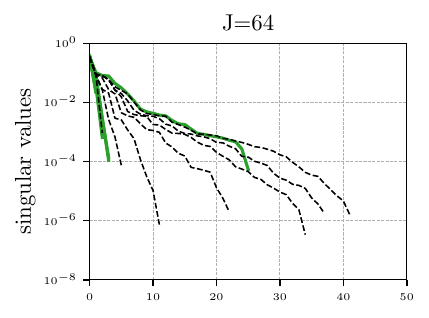}\\[0.4em]
    \includegraphics{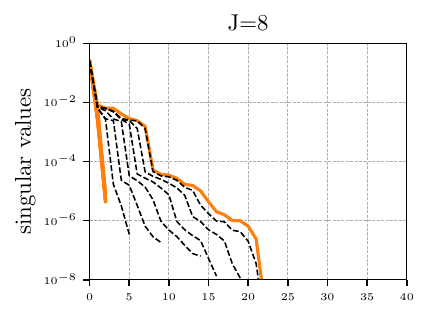}
    \hfill
    \includegraphics{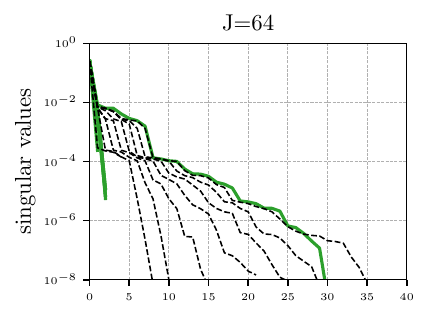}
        \caption{Hierarchical singular values of approximation of $\bu_\varepsilon$ in new format (left) and standard tensor format (right). \textit{Solid lines:} singular values of matricizations $\cM_{\{i\}}(\bu_\varepsilon)$ associated with $i \in \hat{\cI}$ (essentially identical for $i \in \{1,\dots,J\}$, with the line of slower decay corresponding to $i = 0$) and \textit{dashed lines:} singular values of further matricizations in hierarchical representation. Top row: high contrast, bottom row: low contrast.}
    \label{plot:wavelet_sv}
\end{figure}

\subsection{Fourier-type coefficients.}

We also consider a Fourier-type expansion with globally supported spatial functions, given by
\begin{equation}
    a(x,y) = 1 + \sum_{i\geq 1} \theta_i(x) y_i, \quad D = (0,1),
\end{equation}
where 
\begin{equation}
    \theta_i(x) = \frac{c_1}{\bigl(c_2\lambda^{-2} + i^2\bigr)^{(\nu+\frac12)/2}} \sin(\pi i x), \quad i \in \mathbb{N}.
\end{equation}
This provides a Karhunen--Lo\`eve-type test in which the coefficient functions do not have the local multilevel supports exploited by the wavelet construction. In what follows, we set $c_1 = 5.5 \cdot 10^7$, $c_2 = 1$, $\nu = 5.5$, and $\lambda = 0.02$. These parameters yield a short correlation length and a pronounced pre-asymptotic regime with a relatively large number of dominant parameters, as shown in Figure~\ref{plot:fourier_field_theta}. The constant $c_1$ is scaled to ensure the uniform ellipticity condition \eqref{UEA}, while the choice of $\nu$ yields a fast decay of $\norm{\theta_i}_{L_\infty}$ after the pre-asymptotic regime.

\begin{figure}[htbp]
    \centering
    \includegraphics[width=0.48\textwidth]{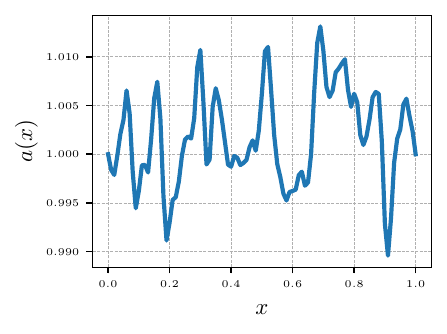}
    \hfill
    \includegraphics[width=0.48\textwidth]{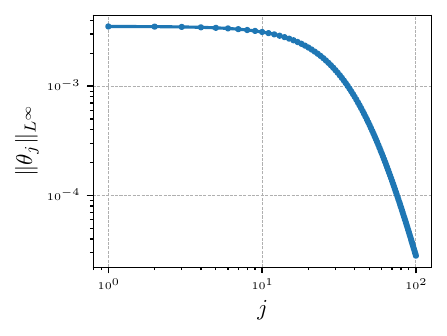}
    \caption{Left: a realization of the Fourier-type coefficient field. Right: decay of the $\norm{\theta_j}_{L_\infty}$.}
    \label{plot:fourier_field_theta}
\end{figure}

We compare several choices of the number $J$ of separated parametric modes. In the present truncated test, $J=100$ corresponds to the fully separated tensor format, whereas smaller values such as $J=10$ and $J=20$ keep the remaining parameters in the sparse expansion mode. Figure~\ref{plot:fourier_set_rank} shows the corresponding residual bounds as functions of the discretization sizes. In contrast to the wavelet-type example, the separation between the curves is less pronounced. Increasing $J$ reduces the burden on the sparse expansion mode only moderately, because the tail parameters are associated with globally supported spatial functions rather than localized multilevel contributions. Thus the Legendre coefficients do not have spatial sparsity as in the wavelet case, and if $\norm{\theta_j}_{L_\infty}$ corresponding to a tail parameter $y_j$ is small, its effect is spread over the full spatial domain, thus requiring a larger spatial discretization than the wavelet-type case.

\begin{figure}[htbp]
    \centering
    \includegraphics{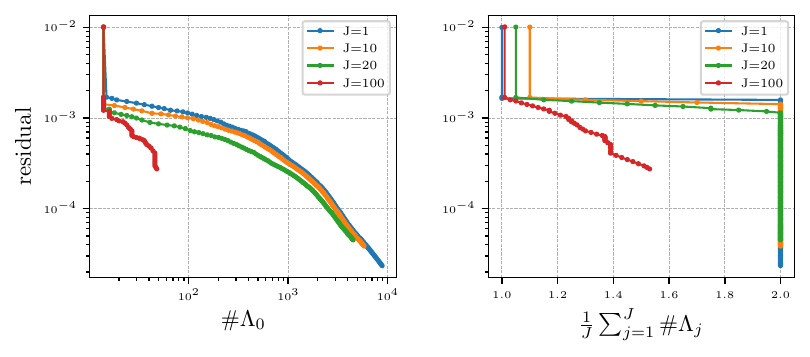}
    \hfill
    \caption{Fourier-type coefficients: computed residual bounds as functions of discretization sizes for different choices of the number $J$ of separated parametric modes.}
    \label{plot:fourier_set_rank}
\end{figure}

The rank behavior in Figure~\ref{plot:fourier_rank} shows a clear advantage of the hybrid choices with fewer separated modes. In particular, the ranks for $J=10$ are substantially smaller than those for the fully separated case $J=100$. This demonstrates that, even for globally supported coefficient functions, avoiding a separation of all tail parameters can significantly reduce the ranks of the computed iterates.

\begin{figure}[htbp]
    \centering
    \includegraphics{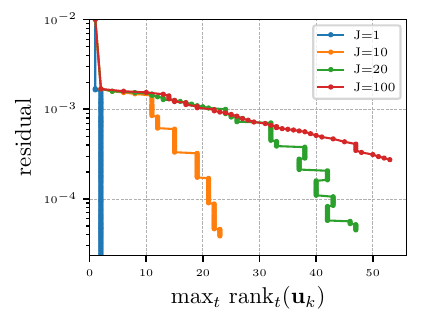}
    \caption{Fourier-type coefficients: maximum hierarchical ranks of iterates $\bu_k$, plotted against the residual error bound $\varepsilon$, for different choices of $J$.}
    \label{plot:fourier_rank}
\end{figure}

The singular values in Figure~\ref{plot:fourier_sv} give the same qualitative picture. The decay for $J=10$ is visibly faster than for $J=100$, and $J=20$ lies between these two cases. The fully separated format introduces many tensor modes whose matricizations have slower singular value decay, whereas the hybrid formats keep the less influential parameters in the sparse expansion mode and thereby obtain more favorable low-rank structure. Thus the Fourier-type test still supports the use of the hybrid format for rank reduction, even though the absence of spatial localization limits the sparse-mode discretization advantage.

\begin{figure}[htbp]
    \centering
    \includegraphics[width=0.32\textwidth]{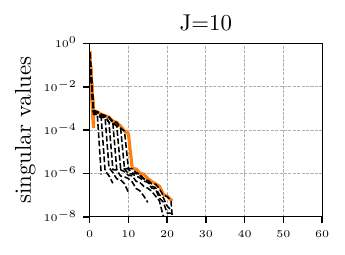}
    \hfill
    \includegraphics[width=0.32\textwidth]{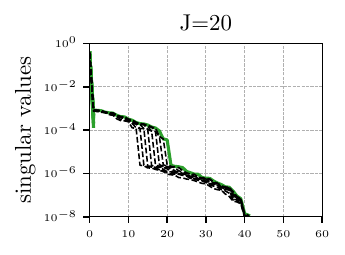}
    \hfill
    \includegraphics[width=0.32\textwidth]{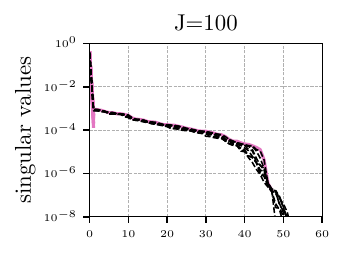}
    \caption{Singular value decay for Fourier-type coefficients for representative choices of dominant block size $J$.}
    \label{plot:fourier_sv}
\end{figure}

\section{Conclusion}

In this work, we have proposed a new approximation format for solutions of parametric problems that combines the advantages of low-rank tensor representations and sparse polynomial expansions. In addition, we have analyzed a new adaptive low-rank Galerkin method that exploits properties of this particular approximation format to achieve quasi-optimal discretizations without intermittent coarsening steps as in previous adaptive low-rank methods, and by the use of soft thresholding of hierarchical tensors provides improved rank bounds.

There are several directions for future work. While in the present work, we have focused on the basic technique of discretization refinement in the low-rank setting, the total computational complexity of the method also depends on the residual approximation technique used for the zeroth tensor mode with tail parameters and spatial variable. For this purpose, it will be of interest to adapt the tree-based wavelet refinement developed in \cite{BachmayrVoulis:22} to our present setting, which can be expected to lead to near-optimal bounds for the total computational complexity of the approach developed here. This also naturally leads to a question of a generalization to spatial discretizations using adaptive finite elements using recent results from \cite{BEEV:24,BEV}.

The performance of the method is tied to specific measures of the approximability of the basis representations of the solution, namely the decay of singular values of matricizations of the coefficient tensor and sparse approximability of the contractions. Our numerical tests show the expected improvement in low-rank approximability compared to fully separated tensor representations by our new approximation format, and they also indicate that the contractions have the expected degree of sparse approximability. An analysis of these measures of approximability will be carried out in a forthcoming work.

Another question that is left open here is the precise relation of the low-rank approximabilities of Galerkin discretizations and exact solutions, which we will consider elsewhere in a more general context.

\section*{Data Availability Statement}
The code used in this paper is available online in the Zenodo repository: \url{https://doi.org/10.5281/zenodo.21101877} \cite{zenodo}.

\bibliographystyle{amsplain}
\bibliography{BYparamalrgalerkin}

\providecommand{\bysame}{\leavevmode\hbox to3em{\hrulefill}\thinspace}
\providecommand{\MR}{\relax\ifhmode\unskip\space\fi MR }
\providecommand{\MRhref}[2]{%
  \href{http://www.ams.org/mathscinet-getitem?mr=#1}{#2}
}
\providecommand{\href}[2]{#2}
\begin{thebibliography}{10}

\bibitem{AliUrban:20}
Mazen Ali and Karsten Urban, \emph{H{T}-{AWGM}: a hierarchical
  {T}ucker-adaptive wavelet {G}alerkin method for high-dimensional elliptic
  problems}, Adv. Comput. Math. \textbf{46} (2020), no.~4, Paper No. 59, 34.

\bibitem{BCM:18}
M.~Bachmayr, A.~Cohen, and G.~Migliorati, \emph{Representations of {G}aussian
  random fields and approximation of elliptic {PDE}s with lognormal
  coefficients}, J. Fourier Anal. Appl. \textbf{24} (2018), 621--649.

\bibitem{BachmayrSchneiderUschmajew:16}
M.~Bachmayr, R.~Schneider, and A.~Uschmajew, \emph{Tensor networks and
  hierarchical tensors for the solution of high-dimensional partial
  differential equations}, Found. Comput. Math. \textbf{16} (2016), no.~6,
  1423--1472.

\bibitem{zenodo}
M.~Bachmayr and H.~Yang, \emph{Julia code for ``{S}parse and low-rank
  approximations of parametric elliptic {PDE}s: the best of both worlds''},
  Zenodo repository, \url{https://doi.org/10.5281/zenodo.21101877}, 2026.

\bibitem{Bachmayr:23}
Markus Bachmayr, \emph{Low-rank tensor methods for partial differential
  equations}, Acta Numer. \textbf{32} (2023), 1--121.

\bibitem{BachmayrCohen:17}
Markus Bachmayr and Albert Cohen, \emph{Kolmogorov widths and low-rank
  approximations of parametric elliptic {PDE}s}, Math. Comp. \textbf{86}
  (2017), no.~304, 701--724.

\bibitem{BachmayrCohenDungSchwab:17}
Markus Bachmayr, Albert Cohen, Dinh D\~{u}ng, and Christoph Schwab, \emph{Fully
  discrete approximation of parametric and stochastic elliptic {PDE}s}, SIAM J.
  Numer. Anal. \textbf{55} (2017), no.~5, 2151--2186.

\bibitem{BachmayrCohenDahmen:18}
Markus Bachmayr, Albert Cohen, and Wolfgang Dahmen, \emph{Parametric {PDE}s:
  sparse or low-rank approximations?}, IMA J. Numer. Anal. \textbf{38} (2018),
  no.~4, 1661--1708.

\bibitem{BachmayrCohenMigliorati:17}
Markus Bachmayr, Albert Cohen, and Giovanni Migliorati, \emph{Sparse polynomial
  approximation of parametric elliptic {PDE}s. {P}art {I}: {A}ffine
  coefficients}, ESAIM Math. Model. Numer. Anal. \textbf{51} (2017), no.~1,
  321--339.

\bibitem{BachmayrDahmen:14}
Markus Bachmayr and Wolfgang Dahmen, \emph{Adaptive near-optimal rank tensor
  approximation for high-dimensional operator equations}, Found. Comput. Math.
  \textbf{15} (2015), no.~4, 839--898.

\bibitem{BEEV:24}
Markus Bachmayr, Martin Eigel, Henrik Eisenmann, and Igor Voulis, \emph{A
  convergent adaptive finite element stochastic galerkin method based on
  multilevel expansions of random fields}, arXiv Preprint arXiv:2403.13770; to
  appear in SIAM Journal on Numerical Analysis, 2024.

\bibitem{BEV}
Markus Bachmayr, Henrik Eisenmann, and Igor Voulis, \emph{Adaptive stochastic
  {G}alerkin finite element methods: {O}ptimality and non-affine coefficients},
  arXiv Preprint arXiv:2503.18704, 2025.

\bibitem{BachmayrSchneider:17}
Markus Bachmayr and Reinhold Schneider, \emph{Iterative methods based on soft
  thresholding of hierarchical tensors}, Found. Comput. Math. \textbf{17}
  (2017), no.~4, 1037--1083.

\bibitem{BachmayrVoulis:22}
Markus Bachmayr and Igor Voulis, \emph{An adaptive stochastic {G}alerkin method
  based on multilevel expansions of random fields: convergence and optimality},
  ESAIM Math. Model. Numer. Anal. \textbf{56} (2022), no.~6, 1955--1992.

\bibitem{BPRR19}
Alex Bespalov, Dirk Praetorius, Leonardo Rocchi, and Michele Ruggeri,
  \emph{Convergence of adaptive stochastic {G}alerkin {FEM}}, SIAM Journal on
  Numerical Analysis \textbf{57} (2019), no.~5, 2359--2382.

\bibitem{CDacta:15}
A.~Cohen and R.~DeVore, \emph{Approximation of high-dimensional parametric
  {PDE}s}, Acta Numer. \textbf{24} (2015), 1--159.

\bibitem{CohenDahmenDeVore:01}
Albert Cohen, Wolfgang Dahmen, and Ronald DeVore, \emph{Adaptive wavelet
  methods for elliptic operator equations: convergence rates}, Math. Comp.
  \textbf{70} (2001), no.~233, 27--75.

\bibitem{CohenDeVore:15}
Albert Cohen and Ronald DeVore, \emph{Approximation of high-dimensional
  parametric {PDE}s}, Acta Numer. \textbf{24} (2015), 1--159.

\bibitem{CohenDeVoreSchwab:10}
Albert Cohen, Ronald DeVore, and Christoph Schwab, \emph{Convergence rates of
  best {$N$}-term {G}alerkin approximations for a class of elliptic s{PDE}s},
  Found. Comput. Math. \textbf{10} (2010), no.~6, 615--646.

\bibitem{CohenDeVoreSchwab:11}
\bysame, \emph{Analytic regularity and polynomial approximation of parametric
  and stochastic elliptic {PDE}'s}, Anal. Appl. (Singap.) \textbf{9} (2011),
  no.~1, 11--47.

\bibitem{CPB:19}
Adam~J. Crowder, Catherine~E. Powell, and Alex Bespalov, \emph{Efficient
  adaptive multilevel stochastic {G}alerkin approximation using implicit a
  posteriori error estimation}, SIAM Journal on Scientific Computing
  \textbf{41} (2019), no.~3, A1681--A1705.

\bibitem{DKLM:15}
Sergey Dolgov, Boris~N Khoromskij, Alexander Litvinenko, and Hermann~G
  Matthies, \emph{Polynomial chaos expansion of random coefficients and the
  solution of stochastic partial differential equations in the tensor train
  format}, SIAM/ASA Journal on Uncertainty Quantification \textbf{3} (2015),
  no.~1, 1109--1135.

\bibitem{EGSZ:14}
Martin Eigel, Claude~Jeffrey Gittelson, Christoph Schwab, and Elmar Zander,
  \emph{Adaptive stochastic {G}alerkin {FEM}}, Computer Methods in Applied
  Mechanics and Engineering \textbf{270} (2014), 247--269.

\bibitem{EMPS:20}
Martin Eigel, Manuel Marschall, Max Pfeffer, and Reinhold Schneider,
  \emph{Adaptive stochastic {G}alerkin {FEM} for lognormal coefficients in
  hierarchical tensor representations}, Numerische Mathematik \textbf{145}
  (2020), 655--692.

\bibitem{EPS:17}
Martin Eigel, Max Pfeffer, and Reinhold Schneider, \emph{Adaptive stochastic
  {G}alerkin {FEM} with hierarchical tensor representations}, Numerische
  Mathematik \textbf{136} (2017), no.~3, 765--803.

\bibitem{GantumurHarbrechtStevenson:07}
Tsogtgerel Gantumur, Helmut Harbrecht, and Rob Stevenson, \emph{An optimal
  adaptive wavelet method without coarsening of the iterands}, Math. Comp.
  \textbf{76} (2007), no.~258, 615--629.

\bibitem{GhanemSpanos:91}
Roger~G. Ghanem and Pol~D. Spanos, \emph{Stochastic finite elements: a spectral
  approach}, Springer-Verlag, New York, 1991.

\bibitem{GhanemSpanos:97}
\bysame, \emph{Spectral techniques for stochastic finite elements}, Arch.
  Comput. Methods Engrg. \textbf{4} (1997), no.~1, 63--100.

\bibitem{Gittelson:13}
Claude~Jeffrey Gittelson, \emph{An adaptive stochastic {G}alerkin method for
  random elliptic operators}, Math. Comp. \textbf{82} (2013), no.~283,
  1515--1541.

\bibitem{Gittelson:14}
\bysame, \emph{Adaptive wavelet methods for elliptic partial differential
  equations with random operators}, Numer. Math. \textbf{126} (2014), no.~3,
  471--513.

\bibitem{Hackbusch:14}
Wolfgang Hackbusch, \emph{Numerical tensor calculus}, Acta Numer. \textbf{23}
  (2014), 651--742. \MR{3202243}

\bibitem{Hackbusch:19}
\bysame, \emph{Tensor spaces and numerical tensor calculus}, Springer Series in
  Computational Mathematics, vol.~56, Springer, Cham, 2019, Second edition.
  \MR{4306578}

\bibitem{HackbuschKuehn:09}
Wolfgang Hackbusch and Stefan K\"{u}hn, \emph{A new scheme for the tensor
  representation}, J. Fourier Anal. Appl. \textbf{15} (2009), no.~5, 706--722.

\bibitem{KhoromskijSchwab11}
Boris~N. Khoromskij and Christoph Schwab, \emph{Tensor-structured {G}alerkin
  approximation of parametric and stochastic elliptic {PDE}s}, SIAM journal on
  scientific computing \textbf{33} (2011), no.~1, 364--385.

\bibitem{KressnerTobler}
Daniel Kressner and Christine Tobler, \emph{Low-rank tensor krylov subspace
  methods for parametrized linear systems}, SIAM Journal on Matrix Analysis and
  Applications \textbf{32} (2011), no.~4, 1288--1316.

\bibitem{Matthies:12}
Hermann~G. Matthies and Elmar Zander, \emph{Solving stochastic systems with
  low-rank tensor compression}, Linear Algebra and its Applications
  \textbf{436} (2012), no.~10, 3819--3838.

\bibitem{Oseledets:11}
I.~V. Oseledets, \emph{Tensor-train decomposition}, SIAM J. Sci. Comput.
  \textbf{33} (2011), no.~5, 2295--2317. \MR{2837533}

\bibitem{SchwabGittelson:11}
Christoph Schwab and Claude~Jeffrey Gittelson, \emph{Sparse tensor
  discretizations of high-dimensional parametric and stochastic {PDE}s}, Acta
  Numer. \textbf{20} (2011), 291--467.

\bibitem{Stevenson:02}
Rob Stevenson, \emph{On the compressibility of operators in wavelet
  coordinates}, SIAM Journal on Mathematical Analysis \textbf{35} (2004),
  no.~5, 1110--1132.

\bibitem{Stevenson:09}
\bysame, \emph{Adaptive wavelet methods for solving operator equations: An
  overview}, Multiscale, Nonlinear and Adaptive Approximation (Berlin,
  Heidelberg) (Ronald DeVore and Angela Kunoth, eds.), Springer Berlin
  Heidelberg, 2009, pp.~543--597.

\end{thebibliography}

\end{document}